\documentclass{amsart}
\usepackage{graphicx}
\usepackage[mathscr]{eucal}
\usepackage{amssymb}
\theoremstyle{plain}
\newtheorem{thm}{Theorem}[section]
\newtheorem{cor}[thm]{Corollary}
\newtheorem{lem}[thm]{Lemma}
\newtheorem{prop}[thm]{Proposition}
\theoremstyle{definition}

\theoremstyle{remark}
\newtheorem{rem}[thm]{Remark}
\newtheorem{ex}[thm]{Example}
\newtheorem{qu}[thm]{Question}

\newcommand{\A}{\mathcal{A}}

\newcommand{\E}{\mathcal{E}}
\newcommand{\U}{\mathcal{U}}
\newcommand{\B}{\mathcal{B}}

\newcommand{\F}{\mathcal{F}}

\newcommand{\mset}{\emptyset}

\newcommand{\pr}{\prime}
\begin{document}
\baselineskip=18pt
\title{A Selection Theorem for Banach Bundles and Applications}
\author{Aldo J. Lazar}
\address{School of Mathematical Sciences\\
        Tel Aviv University\\
        Tel Aviv 69978, Israel}
\email{aldo@post.tau.ac.il}

\date{April 17, 2016}

\subjclass{Primary 46B20; Secondary 55R65, 58B05}

\keywords{Banach bundles, Banach bundle maps, spaces of sections of Banach bundles, lower semi-continuous set-valued maps}

\begin{abstract}

   It is shown that certain lower semi-continuous maps from a paracompact space to the family of closed subsets of the bundle space of a Banach
   bundle admit continuous selections. This generalization of the the theorem of Douady, dal Soglio-Herault, and Hofmann on the fullness of
   Banach bundles has applications to establishing conditions under which the induced maps between the spaces of sections of Banach bundles
   are onto. Another application is to a generalization of the theorem of Bartle and Graves \cite{BG} for Banach bundle maps that are onto their
   images. Other applications of the selection theorem are to the study of the M-ideals of the space of bounded sections begun in \cite{Be} and continued
   in \cite{G}. A
   class of Banach bundles that generalizes the class of trivial Banach bundles is introduced and some properties of these Banach bundles are
   discussed.

\end{abstract}
\maketitle

\newpage
\section{Introduction and basic definitions and notations} \label{S:intro}

The theory of Banach bundles, sometimes presented as continuous fields of Banach spaces, helped to advance the study of unitary representations
of locally compact groups and of $C^*$-algebras. An early example of its usefulness and interest in these fields can be found in \cite{F} and
for later developments we mention \cite{DF}. Other beneficial connections of this theory are with the theory of Banach lattices. A short
exposition on this aspect is given in the comprehensive monograph of Gierz \cite{G} and more recent results appeared in \cite{Gut1}. The theory
of Banach bundles provides an association of topology, the geometry of Banach spaces and operator theory, as exemplified in \cite[Chapters 15
and 16]{G} and \cite{GK}.

We prove in this paper a generalization of Michael's well known selection theorem, \cite[Theorem 3.2'']{M}, in the context of Banach bundles.
This result, Theorem \ref{T:selection}, also generalizes the theorems of Douady, L. dal Soglio-Herault, and K. H. Hofmann, see \cite[Appendix
C]{DF}, on the existence of sufficiently many continuous cross-sections in a Banach bundle. A particular case of Theorem \ref{T:selection} is
\cite[Proposition 15.13]{G}.

In Section \ref{S:maps} we consider Banach bundle maps (definition will follow) and the induced maps between the spaces of sections. We apply
the selection theorem to establishing some conditions under which these later maps are onto their images. This investigation led to examining
when a map between two Banach bundles is open. The section ends with the generalization of a theorem of Bartle and Graves on the existence of
continuous right inverses for maps between Banach spaces. Here the context is of course maps between Banach bundles and we follow Michael
\cite{M} in using the selection theorem for deriving the existence of such an inverse.

Alfsen and Effros developed in \cite{AE} a structure theory for Banach spaces in which certain subspaces named M-ideals play a crucial role. The
M-ideals of the Banach space of sections of a Banach bundle whose base space is compact Hausdorff were investigated in depth by Behrends
\cite{Be} and Gierz \cite{G}. In Section \ref{S:M-ideals} we extend this investigation to the case when the base space is locally compact
Hausdorff and the sections vanish at infinity.

A certain class of continuous Banach bundles is introduced in Section \ref{S:uniform}; the bundles in this class are called locally uniform.
Every locally trivial Banach bundle is locally uniform but the converse is false. The locally uniform Banach bundles enjoy some nice properties.
For instance we show in Proposition \ref{P:qcont2} that the quotient of a locally uniform Banach bundle by a locally uniform Banach subbundle is a
continuous Banach bundle.

Some topological properties of the bundle space are treated in the Appendix. Conditions which insure that the bundle space is a Baire space are
given. Also its paracompactness and its metrizability are discussed there.

Throughout of this paper the field of scalars is denoted by $\mathbb{K}$; it can be the field of the real numbers or the field of the complex
numbers. A paracompact space is always considered to be Hausdorff. A topological space is locally paracompact if it has a cover with open
subsets whose closures are paracompact. A map $\Phi$ from a topological space $T_1$ to the family of subsets of a topological space $T_2$ is
called lower semi-continuous if for every open non-void subset $O$ of $T_2$ the set $\{t\in T_1\mid O\cap \Phi(t)\neq \mset\}$ is open in $T_1$.

For a Banach space $X$, $x\in X$, and $r$ a positive number we denote $B(x,r) := \{y\in X\mid \|y - x\| < r\}$ and by $X_r$ the closed ball of
$X$ whose center is at the origin with radius $r$. A closed subspace $Y$ of the Banach space $X$ is called an M-ideal of $X$ if its polar $Z_1$
in the dual of $X$ has a closed complement $Z_2$ such that $\|z_1 + z_2\| = \|z_1\| + \|z_2\|$ for every $z_i\in Z_i$, $i = 1,2$. A closed
subspace $Y$ of $X$ is an M-ideal if and only if ithas the 3-ball property: if $\cap_{i=1}^3 B(x_i,r_i)\neq \mset$ and $B(x_i,r_i)\cap Y\neq
\mset$, $1\leq i\leq 3$, then $Y\cap \cap_{i=1}^3 B(x_i,r_i)\neq \mset$.

By a Banach bundle $\xi := (\E,p,T)$ we shall mean what in \cite[pp. 8,9]{DG} is called an (H)Banach bundle. That is $x\to \|x\|$ is upper
semi-continuous on the bundle space $\E$, see also \cite[p.21]{G}. We shall always suppose that the base space $T$ is Hausdorff. The fiber
$p^{-1}(t)$ over $t\in T$ will be sometimes denoted $\E(t)$. The origin of the Banach space $\E(t)$ is denoted $0_t$ but if there can be no
confusion just $0$ will be used. The bundle $\xi$ is called a continuous Banach bundle if $x\to \|x\|$ is continuous on $\E$. The bundle space
of a continuous Banach bundle is always Hausdorff, see \cite[Proposition 16.4]{G}. If $\E^{\pr}$ is a subset of $\E$ such that
$p\mid_{\E^{\pr}}$ is onto $T$ and open and each $p^{-1}(t)\cap \E^{\pr}$ is a closed subspace of $\E(t)$ then $\xi^{\pr} :=
(\E^{\pr},p\mid_{\E^{\pr}},T)$ is a Banach bundle called a Banach subbundle of $\xi$. One can define an equivalence relation on $\E$ as follows:
$x_1,x_2$ are equivalent if $p(x_1) = p(x_2)$ and $x_1 - x_2$ belongs to $p^{-1}(p(x_1))$. The quotient space, denoted $\E/\E^{\pr}$, is the
union of all the quotient Banach spaces $\E(t)/\E^{\pr}(t)$, $t\in T$. If we denote by $\tilde{p}$ the obvious map of $\E/\E^{\pr}$ onto $T$
then $\eta := (\E/\E^{\pr},\tilde{p},T)$ is a Banach bundle by \cite[Chapter 9]{G}. The quotient map $q : \E\to \E/\E^{\pr}$ is open, see
\cite[9.4]{G}.

Let $\xi := (\E,p,T)$ be a Banach bundle. A continuous function $f : T\to \E$ is called a section of $\xi$ if $p(f(t)) = t$ for every $t\in T$.
The linear space of all the sections of $\xi$ is denoted $\Gamma(\xi)$; its subspace of all the bounded sections is denoted $\Gamma_b(\xi)$.
This is a Banach space, the norm of $f\in \Gamma_b(\xi)$ being $\|f\| := \sup_{t\in T} \|f(t)\|$. If $T$ is locally compact Hausdorff then the
subspace $\Gamma_0(\xi)$ of $\Gamma_b(\xi)$ consisting of all the sections that vanish at infinity is of interest.

A full Banach bundle is a bundle such that for each $x\in \E$ there exists a section $f\in \Gamma(\xi)$ satisfying $f(p(x)) = x$. In this
definition one can replace $\Gamma(\xi)$ by $\Gamma_b(\xi)$ as observed in \cite[p. 14]{DG} and even by $\Gamma_0(\xi)$ if $T$ is locally
compact. For $f\in \Gamma(\xi)$, $V$ an open subset of $T$, and $a > 0$ the open subset
\[
 U(f,V,a) := \{x\in \E\mid p(x)\in V, \ \|x - f(p(x))\| < a\}
\]
of $\E$ is called a tube. If $\xi$ is a full Banach bundle then the family of all the tubes forms a base for the topology of $\E$ by \cite[p.
10]{DG}.

A Banach bundle map from $\xi_1 := (\E_1,p_1,T)$ to $\xi_2 := (\E_2,p_2,T)$ is a continuous map $\varphi : \E_1\to \E_2$ such that
$\varphi(p_1^{-1}(t))\subset p_2^{-1}(t)$ for every $t\in T$ and the restriction of $\varphi$ to each fiber $p_1^{-1}(t)$ is linear. The map
$\varphi$ is called an (isometric) isomorphism if $\varphi(\E_1) = \E_2$ and $\varphi$ is an isometry on every fiber.

\section{The selection theorem} \label{S:selection}

In this section we shall state and prove the selection theorem for set valued maps into a Banach bundle and we shall discuss some of its
applications. The proof mimics closely the proofs given in \cite[Appendix C]{DF} and \cite{M}. A sequence of lemmas prepares the proof of the
theorem itself.

From now on in this section $\xi := (\E,p,T)$ is a fixed Banach bundle. Following \cite[Appendix C]{DF} we shall say that a subset $U$ of $\E$
is $\varepsilon$-thin, $\varepsilon > 0$, if $\|x - x^{\pr}\| < \varepsilon$ whenever $x, x^{\pr} \in \E$ and $p(x) = p(x^{\pr})$.

\begin{lem} \label{L:seg}

   Let $t\in T$, $x_1, x_2\in p^{-1}(t)$ with $\|x_1 - x_2\| < \varepsilon$. There is an open $\varepsilon$-thin set $U$ that contains the
   segment $[x_1,x_2]$.

\end{lem}

\begin{proof}

   With $D := \{(x,y)\in \E\times \E \mid p(x) = p(y)\}$, the set $\{(x,y) \mid \|x - y\| < \varepsilon\}$ is open since $(x,y)\to \|x - y\|$ is
   upper semi-continuous. There is an open subset $O$ of $\E\times \E$ such that
   \[
    [x_1,x_2]\times [x_1,x_2]\subset \{(x,y) \mid \|x - y\|\} = O\cap D.
   \]
   By \cite[p. 171]{Mu} there is an open set $U\subset \E$ such that
   \[
    [x_1,x_2]\times [x_1,x_2]\subset U\times U\subset O.
   \]
   Obviously $U$ is $\varepsilon$-thin.

\end{proof}

We enlarge now our setting with a completely regular topological space $S$ and a continuous open map $\pi$ of $S$ onto $T$. A function $f : S\to
\E$ is called admissible if $p\circ f = \pi$. Let $\varepsilon > 0$. A function $f :S\to \E$ is called $\varepsilon$-continuous at $s\in S$ if
there are a neighbourhood $V$ of $s$ and an $\varepsilon$-thin neighbourhood $U$ of $f(s)$ such that $f(V)\subset U$. If $f$ is
$\varepsilon$-continuous at all the points of $S$ then it is called shortly $\varepsilon$-continuous.

\begin{lem} \label{L:all}

   If an admissible function $f$ is $\varepsilon$-continuous for all $\varepsilon > 0$ then it is continuous.

\end{lem}

\begin{proof}

   Let $s\in S$ and and suppose that $V_n$ is a neighbourhood of $s$ and $U_n$ is an $1/n$-thin neighbourhood of $f(s)$, with
   $f(V_n)\subset U_n$, $n\in \mathbb{N}$. We have $f(V_n\cap V_{n+1})\subset U_n\cap U_{n+1}$ so there is no loss of generality if we suppose
   that $\{U_n\}$ is decreasing. Let now $W$ be a neighbourhood of $f(s)$. By \cite[p. 10]{DG} there are a natural number $n$ and a
   neighbourhood $V$ of $\pi(s)$ such that $U_n\cap p^{-1}(V)\subset W$. The neighbourhood $V\cap \pi(V_n)$ of $p(f(s)) = \pi(s)$ satisfies
   $U_n\cap p^{-1}(V\cap \pi(V_n))\subset W$. The neighbourhood $V_n\cap \pi^{-1}(V)$ os $s$ satisfies $f(V_n\cap \pi^{-1}(V))\subset W$.
   Indeed, if $s^{\pr}\in V_n\cap \pi^{-1}(V)\subset V_n$ then $f(s^{\pr})\in U_n$. Moreover, $p(f(s^{-1})) = \pi(s^{\pr})\in \pi(V_n)\cap V$;
   hence $f(s^{\pr})\in p^{-1}(\pi(V_n)\cap V)\subset p^{-1}(V)$. Thus $f(s^{\pr})\in U_n\cap p^{-1}(V)\subset W$ and the continuity of $f$ at
   $s$ is established.

\end{proof}

\begin{lem} \label{L:open}

   Let $G$ be an open subset of $S$ and $\{U_i\}_{i=1}^n$ be open subsets of $\E$ such that $p(U_i) = \pi(G)$, $1\leq i \leq n$. Suppose that
   $\{\varphi_i\}_{i=1}^n$ are continuous scalar functions on $G$ with $\varphi_1(s)\neq 0$ for each $s\in G$. Then the set
   \[
    W := \{\sum\varphi_i(s)x_i \mid s\in G, x_i\in U_i, p(x_i) = \pi(s), 1\leq i\leq n\}
   \]
   is open in $\E$.

\end{lem}

\begin{proof}

   Consider the set
   \[
    Z := \{(s,x_1\ldots, x_n) \mid s\in G, x_i\in \E, p(x_i) = \pi(s), 1\leq i\leq n\}
   \]
   endowed with the relative topology inherited from $G\times p^{-1}(\pi(G))^n$. A basis for this topology consists of all the sets of the form
   $O := (V\times V_1\times \dots V_n)\cap Z$ where $V\subset G$, $V_i\subset p^{-1}(\pi(G))$, $1\leq i\leq n$, are open. Thus
   \begin{equation} \label{E:basis}
    O := \{(s,x_1,\ldots x_n)\mid s\in V, x_i\in V_i, p(x_i) = \pi(s), 1\leq i\leq n\}.
   \end{equation}
   The map $\alpha$ of $Z$ onto itself given by $\alpha(s,x_1,\ldots x_n) := (s,\sum\varphi_i(s)x_i, x_2,\ldots x_n)$ is a homeomorphism. Indeed
   it is continuous and its inverse maps $(s,y_1,\ldots y_n)$ to $(s,1/\varphi_1(s)(y_1-\sum_{i=2}^n\varphi_i(s)y_i),y_2,\ldots y_n)$. The map
   $\beta$ of $Z$ into $\E$ given by $\beta(s,x_1,\ldots x_n) := x_1$ is open. Indeed, if $O$ is as in~\eqref{E:basis} then
   \[
    \beta(O) = V_1\cap p^{-1}(\pi(V)\cap \cap_{i=1}^n p(V_i))
   \]
   that is open. Now $W = \beta(\alpha((G\times U_1,\ldots U_n)\cap Z))$ is obviously open.

\end{proof}

\begin{lem} \label{L:comb}

   Suppose $\varphi_i$, $1\leq i\leq n$, are continuous functions from $S$ to $[0,1]$ with $\sum\varphi_i = 1$. Let $\varepsilon > 0$. If
   $f_i$, $1\leq i\leq n$ are admissible and $\varepsilon$-continuous at $s_0\in S$ then $f := \sum\varphi_if_i$ is admissible and
   $\varepsilon$-continuous at $s_0$.

\end{lem}

\begin{proof}

   We have to prove only that $f$ is $\varepsilon$-continuous at $s_0$. There is no loss of generality if we suppose $\varphi_1(s_0)\neq 0$. There exist $\varepsilon$-thin open neighbourhoods $U_i$ of $f_i(s_0)$
   and open neighbourhoods $V_i$ of $s_0$ such that $f_i(V_i)\subset U_i$, $1\leq i\leq n$. By taking $\cap_{i=1}^n V_i$ and further reducing it if
   needed one gets an open neighbourhood $V$ of $s_0$ such that $f_i(V)\subset U_i$, $1\leq i\leq n$, and $\varphi(s) > 0$ if $s\in V$. From
   $f_i(V)\subset U_i$ and the admissibility of $f_i$ we get $f_i(V)\subset p^{-1}(\pi(V))\cap U_i$ and $p(p^{-1}(\pi(V))\cap U_i) = \pi(V)$,
   $1\leq i\leq n$. The set $p^{-1}(\pi(V))\cap U_i$ is an open neighbourhood of $f_i(s_0)$. The set
   \[
    W := \{\sum\varphi_i(s)x_i \mid s\in V, x_i\in p^{-1}(\pi(V))\cap U_i, p(x_i) = \pi(s), 1\leq i\leq n\}
   \]
   is open in $\E$ by Lemma \ref{L:open} and $f(s_0)\in W$. By the above $f(V)\subset W$ and it remains to show that $W$ is $\varepsilon$-thin. To
   this end, let $x^{\pr}, x^{\pr \pr}\in W$ with $p(x^{\pr}) = p(x^{\pr \pr}) = \pi(s)$ for some $s\in V$. Then $x^{\pr} =
   \sum\varphi_i(s)x_i^{\pr}$ with $x_i^{\pr}\in U_i$, $p(x_i^{\pr}) = \pi(s)$ and $x^{\pr \pr} = \sum\varphi_i(s)x_i^{\pr \pr}$ with
   $x_i^{\pr \pr}\in U_i$, $p(x_i^{\pr \pr}) = \pi(s)$, $1\leq i\leq n$. Each $U_i$ is $\varepsilon$-thin so $\|x^{\pr} - x^{\pr \pr}\| <
   \varepsilon$ and we conclude that $W$ is $\varepsilon$-thin.

\end{proof}

\begin{lem} \label{L:exist}

   Let $x_0\in \E$, $s_0\in S$ with $\pi(s_0) = p(x_0)$. There exists an $\varepsilon$-continuous admissible function $f : S\to \E$ such that
   $f(s_0) = x_0$.

\end{lem}

\begin{proof}

   Let $U$ be an $\varepsilon$-thin open neighbourhood of $x_0$ and set $V := \pi^{-1}(p(U))$ which is an open neighbourhood of $s_0$. Define $g
   : V\to \E$ so that $g(s_0) = x_0$ and $g(s)\in U$ with $p(g(s)) = \pi(s)$ for $s\in V$. Let $W$ be an open neighbourhood of $s_0$ such that
   $\overline{W}\subset V$ and $\varphi : S\to [0,1]$ a continuous function such that $\varphi(s_0) = 1$ and $\varphi\mid_{S\setminus
   \overline{W}} \equiv 0$. Set
   \begin{equation}
    f(s) :=
    \begin{cases}
       \varphi(s)g(s), &\text{if $s\in V$,}\\
       0, &\text{if $s\notin V$}.
    \end{cases}
   \end{equation}

Then on $V$, $f = \varphi g + (1 - \varphi)0$. Thus $f$ is $\varepsilon$-continuous on $V$ by Lemma \ref{L:comb} and trivially
$\varepsilon$-continuous on $S\setminus \overline{W}$. Clearly $f$ is an admissible function.

\end{proof}

From here on, until the end of the proof of Theorem \ref{T:selection}, $S$ is a paracompact space and $\Phi$ is a lower semi-continuous map from
$S$ to the family of non-void closed subsets of $\E$ such that $\Phi(s)$ is a convex subset of $p^{-1}(\pi(s))$ for every $s\in S$.

\begin{lem} \label{L:app}

   Given $s_0\in S$ and $\varepsilon > 0$ there are an $\varepsilon$-continuous admissible function $f$ and a neighbourhood $V$ of $s_0$ such
   that $B(f(s),\varepsilon)\cap \Phi(s)\neq \mset$ for every $s\in S$.

\end{lem}

\begin{proof}

   Let $x_0\in \Phi(s_0)$ and $f$ be an $\varepsilon$-continuous admissible function such that $f(s_0) = x_0$ as given by Lemma \ref{L:exist}.
   There are an $\varepsilon$-thin neighbourhood $U$of $x_0$ and a neighbourhood $V_1$ of $s_0$ such that $f(V_1)\subset U$. The set $V_2 :=
   \{s\in S\mid U\cap \Phi(s)\neq \mset\}$ is an open neighbourhood of $s_0$. Set $V := V_1\cap V_2$. Then, for each $s\in V$ there is $x\in
   U\cap \Phi(s)$ and $\|f(s) - x\| < \varepsilon$ since $f(s)\in U$ and $U$ is an $\varepsilon$-thin set.

\end{proof}

Now we have reached the main approximation lemma. Its statement consists of two assertions with similar proofs.

\begin{lem} \label{L:app1}
   Let $\varepsilon > 0$.
   \begin{itemize}
      \item[(i)] There is an $\varepsilon$-continuous admissible function $g$ such that \linebreak $B(g(s),\varepsilon)\cap \Phi(s)\neq \mset$ for every
                 $s\in S$.
      \item[(ii)] Given an $\varepsilon$-continuous admissible function $f$ such that $B(f(s),\varepsilon)\cap \Phi(s)\neq \mset$ for every $s\in
      S$ there is an $\varepsilon/2$-continuous admissible function $g$ such that $\|f(s) - g(s)\| < 2\varepsilon$ and
      $B(g(s),\varepsilon/2)\cap \Phi(s)\neq \mset$ for every $s\in S$.
   \end{itemize}

\end{lem}

\begin{proof}

   We begin with the proof of (ii). Given $f$ as above, the first step of the proof is to show that for every $s\in S$ there are an open
   neighbourhood $V$ and an $\varepsilon/2$-continuous admissible function $h$ such that $B(h(s^{\pr}),\varepsilon/2)\cap \Phi(s^{\pr})\neq \mset$
   and $\|f(s^{\pr}) - h(s^{\pr})\| < 2\varepsilon$ for each $s^{\pr}\in V$. Let $s\in S$ and choose $x\in B(f(s),\varepsilon)\cap \Phi(s)$. There
   is an $\varepsilon$-thin open neighbourhood $U_1$ of $f(s)$ and an open neighbourhood $V_1$ of $s$ such that $f(V_1)\subset U_1$. Let $y = 3/4x +
   1/4f(s)$. Lemma \ref{L:exist} yields an $\varepsilon/4$-continuous admissible function $h$ such that $h(s) = y$. We are going to show that $h$
   satisfies our claim.

   There are an $\varepsilon/4$-thin open neighbourhood $U_2$ of $y$ and an open neighbourhood $V_2$ of $s$ such that
   $h(V_2)\subset U_2$. By using Lemma \ref{L:seg} we get an $\varepsilon/4$-thin open set $U_3$ that contains the segment $[y,x]$ and an open
   $3\varepsilon/4$-thin set $U_4$ that contains the segment $[y,f(s)]$. Denote $V^{\pr} := \{s^{\pr}\mid U_3\cap \Phi(s^{\pr})\neq \mset\}$, an
   open neighbourhood of $s$ since $x\in U_3\cap \Phi(s)$. Set now
   \[
    V := \pi^{-1}(p(U_2\cap U_3\cap U_4))\cap \pi^{-1}(p(U_1\cap U_4))\cap V_1\cap V_2\cap V^{\pr}.
   \]
   We have $s\in \pi^{-1}(p(U_2\cap U_3\cap U_4))$ since $y\in U_2\cap U_3\cap U_4$ and $p(y) = p(f(s)) = \pi(s)$. Moreover, $f(s)\in U_1\cap
   U_4$ hence $s\in \pi^{-1}(p(U_1\cap U_4))$. We conclude that $V$ is an open neighbourhood of $s$. If $s^{\pr}\in V\subset V_1$ then $f(s^{\pr})\in U_1$
   and $\pi(s^{\pr})\in p(U_1\cap U_4)$. Thus there is $x_1\in U_1\cap U_4$ with $p(x_1) = \pi(s^{\pr})$ and
   \begin{equation} \label{E:first}
      \|f(s^{\pr}) - x_1\| < \varepsilon
   \end{equation}
   since $U_1$ is $\varepsilon$-thin. For $s^{\pr}\in V$ we have $\pi(s^{\pr})\in
   p(U_2\cap U_3\cap U_4)$ hence there is $x_2\in U_2\cap U_3\cap U_4$ such that $p(x_2) = \pi(s^{\pr})$. From $s^{\pr}\in V\subset V_2$ we get
   $h(s^{\pr})\in U_2$ and
   \begin{equation} \label{E:second}
      \|h(s^{\pr}) - x_2\| < \varepsilon/4.
   \end{equation}
   The set $U_4$ is $3/4\varepsilon$-thin so
   \begin{equation} \label{E:third}
      \|x_1 - x_2\| < 3/4\varepsilon.
   \end{equation}
   Finally, if $s^{\pr}\in V\subset V^{\pr}$ then there is $x_3\in U_3\cap \Phi(s^{\pr})$ and we have
   \begin{equation} \label{E:forth}
      \|x_2 - x_3\| < \varepsilon/4
   \end{equation}
   since $U_3$ is $\varepsilon/4$-thin. From~\eqref{E:second} and~\eqref{E:forth} we get
   \[
    \|h(s^{\pr}) - x_3\|\leq \|h(s^{\pr}) - x_2\| + \|x_2 - x_3\| < \varepsilon/2.
   \]
   Therefore, if $s^{\pr}\in V$ then $B(h(s^{\pr}),\varepsilon/2)\cap \Phi(s^{\pr})\neq \mset$. Also, if $s^{\pr}\in V$ we get from~\eqref{E:second}, \eqref{E:third} and~
   \eqref{E:first} that
   \[
    \|h(s^{\pr}) - f(s^{\pr})\|\leq \|h(s^{\pr}) - x_2\| + \|x_2 - x_1\| + \|x_1 - f(s^{\pr})\| < \varepsilon/4 + 3\varepsilon/4 + \varepsilon =
                                         2\varepsilon
   \]
   and the claim is established.

   As mentioned above, what has been done up to this point is irrelevant for (i). To obtain the proof of (i) from what follows one
   has to ignore every affirmation about the function $f$.

   We know now that there exist an open covering $\{V_{\alpha}\}_{\alpha\in \A}$ of $S$ and admissible $\varepsilon/2$-continuous functions
   $h_{\alpha} : S\to \E$ such that $B(h_{\alpha}(s),\varepsilon/2)\cap \Phi(s)\neq \mset$ and $\|f(s) - h_{\alpha}\| < 2\varepsilon$ for every $s\in
   V_{\alpha}$. For proving (i) one obtains this from Lemma \ref{L:app}. We may and shall suppose that the covering of $S$ is locally finite.
   Let $\{\varphi_{\alpha}\}$ be a partition of unity subordinated to $\{V_{\alpha}\}$. and define $g := \sum_{\alpha}
   \varphi_{\alpha}h_{\alpha}$. Each point $s\in S$ has a neighbourhood $O_s$ such that there exists a finite set $\F_s\subset \A$ with the
   property that $O_s\cap V_{\alpha} = \mset$ is $\alpha\notin \F_s$. Thus $g$ is well defined, admissible and $\varepsilon/2$-continuous by
   Lemma \ref{L:comb}. We have $B(g(s),\varepsilon/2)\cap \Phi(s)\neq \mset$ for every $s\in S$. Indeed, choose $z_{\alpha}\in
   B(h_{\alpha}(s),\varepsilon/2)\cap \Phi(s)$ for $\alpha\in \F_s$. Then $z := \sum_{\alpha\in \F_s} z_{\alpha}\in \Phi(s)$ and
   \[
    \|g(s) - z\|\leq \sum_{\alpha\in \F_s} \varphi_{\alpha}(s)\|h_{\alpha}(s) - z_{\alpha}\| < \varepsilon/2.
   \]
   Furthermore,
   \[
    \|g(s) - f(s)\|\leq \|\sum_{\alpha\in \F_s}\varphi_{\alpha}(s)h_{\alpha}(s) - \sum_{\alpha\in \F_s}\varphi_{\alpha}(s)f(s)\|\leq
    \sum_{\alpha\in \F_s}\varphi_{\alpha}(s)\|h_{\alpha}(s) - f(s)\| < 2\varepsilon
   \]
   and the proof is complete.

\end{proof}

\begin{lem} \label{L:cont}

  Let $\{f_n\}$ be a sequence of admissible $\varepsilon$-continuous ($\varepsilon > 0$) functions, uniformly convergent on $S$. Then its limit
  $ f$ is admissible and $2\varepsilon$-continuous.

\end{lem}

\begin{proof}

   Clearly $f$ is admissible; we have only to prove that it is $2\varepsilon$-continuous. Let $n$ be such that $\|f(s) - f_n(s)\| < \varepsilon/2$
   for every $s\in S$. For $s_0\in S$ there are a neighbouhood $V$ of $s_0$ and a neighbourhood $U^{\pr}$ of $f_n(s_0)$ such that
   $f_n(V)\subset U^{\pr}$. Now, $f_n$ is admissible so $\pi(V)\subset U^{\pr}$. Put $U := U^{\pr}\cap p^{-1}(\pi(V))$ and
   \[
    W := \{x + y \mid p(x) = p(y)\in \pi(V), \|x\| < \varepsilon/2, y\in U\}.
   \]
   Then Lemma \ref{L:open} yields that $W$ is an open neighbourhoo of $f(s_0)$ and it is easily seen that it is $2\varepsilon$-thin. If $s\in V$
   then $f(s) = f_n(s) + (f(s) - f_n(s))\in W$. We conclude that $f$ is $2\varepsilon$-continuous at $s_0$.

\end{proof}

In all the applications but one of the theorem that follows the setup will be somewhat simpler; we shall have $S = T$ and $\pi$ will be the
identity map of $T$.

\begin{thm} \label{T:selection}

   Let $\xi := (p,\E,T)$ be a Banach bundle, $S$ a paracompact topological space, $\pi$ a continuous open map of $S$ onto $T$ and $\Phi$ a lower
   semi-continuous map from $S$ to the family of non-empty closed subsets of $\E$ such that $\Phi(s)$ is a convex subset of the fiber
   $p^{-1}(\pi(s))$ for every $s\in S$. Then there exists a continuous admissible function $f : S\to \E$ such that $f(s)\in \Phi(s)$ for every $s\in
   S$. Moreover, if $s_0\in S$ and $x_0\in \Phi(s_0)$ then $f$ can be chosen so that $f(s_0) = x_0$.

\end{thm}

\begin{proof}

   By Lemma \ref{L:app1} (i) there exists an $1/2$-continuous admissible function $f_1$ such that $B(f_1(s),1/2)\cap \Phi(s)\neq \mset$ for
   every $s\in S$. We proceed now by induction. Suppose that admissible functions $\{f_i\}_{i=1}^n$ are given such that each $f_i$ is
   $1/2^i$-continuous, satisfies $B(f_i(s),1/2^i)\cap \Phi(s)\neq \mset$ and $\|f_i(s) - f_{i+1}(s)\| < 1/2^{i-1}$ for every $s\in S$ and $1\leq i\leq
   n  - 1$. By Lemma \ref{L:app1} (ii) there exists an $i/2^{n+1}$-continuous admissible function $f_{n+1}$ such that
   $B(f_{n+1}(s),1/2^{n+1})\cap\Phi(s)\neq \mset$ and $\|f_n(s) - f_{n+1}(s)\| < 1/2^{n-1}$ for every $s\in S$. The sequence $\{f_n\}$ is
   uniformly convergent on $S$. The limit $f$ is admissible and $1/2^n$-continuous for $n\geq 1$ by Lemma \ref{L:cont}. Thus $f$ is continuous
   on $S$ by Lemma \ref{L:all}. Obviously $f(s)\in \Phi(s)$ for each $s\in S$.

   The second assertion of the theorem follows from the first by considering the lower semi-continuous set valued map
   \begin{equation}
    \Phi^(s) :=
    \begin{cases}
       \Phi(s),\  s\neq s_0,\\
       x_0,\  s = s_0.
    \end{cases}
   \end{equation}

\end{proof}

Theorem \ref{T:selection} is a generalization of \cite[Proposition 15.13]{G} where the case $S = T$, a compact Hausdorff space, $\pi$ being the
identity map of $T$, is discussed. By considering a trivial Banach bundle one can deduce from Theorem \ref{T:selection} the existence of the
selection part of Michael's \cite[Theorem 3.2'']{M}. Another known result that can be deduced from Theorem \ref{T:selection} is the observation
on p. 15 of \cite{DG} that partial section over closed subsets of the base space can be extended to sections over the entire base space. We
state below the precise result.

\begin{cor} \label{C:ext}

   Let $\xi := (p,\E,T)$ be a Banach bundle with paracompact base space, $A$ a closed subset of $T$ and $\varphi : A\to \E$ a section of
   $\xi\mid A$. Then there exists a section $\tilde{\varphi} : T\to \E$ of $\xi$ that extends $\varphi$. Moreover, if $\varphi$ is bounded then
   there exists a bounded extension $\tilde{\varphi}$. If $\xi := (\E,p,T)$ is a Banach bundle with locally compact Hausdorff base space, $A$ a
   compact subset of $T$ and $\varphi$ a section of $\xi\mid A$ then there is a section $\tilde{\varphi}\in \Gamma_0(\xi)$ that extends
   $\varphi$.

\end{cor}

\begin{proof}

   The first claim is obtained by applying Theorem \ref{T:selection} to the set valued map
   \begin{equation}
    \Phi(t) :=
    \begin{cases}
       \{\varphi(t)\},\  t\in A,\\
       p^{-1}(t),\  t\in T\setminus A.
    \end{cases}
   \end{equation}
   To get a bounded extension in case $\varphi$ is bounded one multiplies a section that extends $\varphi$ by a bounded continuous scalar
   valued function that is constantly equal to $1$ on $A$. Finally, when $T$ is locally compact Hausdorff and $A$ is compact one extends
   $\varphi$ to a section $\varphi^{\pr}$ of $\xi\mid \overline{V}$, $V$ being an open neighbourhood of $A$ with compact closure. With $f : T\to
   [0,1]$ a continuous function such that $f\mid A \equiv 1$ and $f\mid T\setminus V \equiv 0$ one defines
   \begin{equation}
    \tilde{\varphi}(t) :=
    \begin{cases}
       f(t)\varphi^{\pr}(t),\  t\in V\\
       0_t,\  t\in T\setminus V.
    \end{cases}
   \end{equation}

\end{proof}

Of course, one can deduce from Corollary \ref{C:ext} the well known theorems of A. Douady, L. dal Soglio-Herault and K. H. Hofmann (\cite[pp.
640-641]{DF}) on the fullness of Banach bundles with paracompact or locally compact Hausdorff base spaces.

\section{Banach bundle maps} \label{S:maps}

We shall now discuss Banach bundle maps and the way they affect various spaces of sections. We shall end this chapter with an extension of a
well known result of Bartle and Graves \cite{BG} to Banach bundle maps.

Let $\xi_1 := (\E_1,p_1,T)$ and $\xi_2 := (\E_2,p_2,T)$ be Banach bundles. Recall that a Banach bundle map from $\xi_1$ to $\xi_2$ is a
continuous map $\varphi : \E_1\to \E_2$ such that $p_2\circ \varphi = p_1$ and which acts linearly on each fiber of $\xi_1$. The restriction of
$\varphi$ to $p_1^{-1}(t)$, $t\in T$, is a bounded linear operator into $p_2^{-1}(t)$ which we shall denote by $\varphi_t$. The Banach bundle
map $\varphi$ induces a linear map $\tilde{\varphi} : \Gamma(\xi_1)\to \Gamma(\xi_2)$ as follows:
\[
 \tilde{\varphi}(f)(t) := \varphi(f(t)), \ f\in \Gamma(\xi_1), \ t\in T.
\]
This map $\tilde{\varphi}$ need not be onto $\Gamma(\xi_2)$ even if $\varphi$ is onto $\E_2$. To illustrate this we shall present an example
inspired by one in \cite{B}. In order to have $\tilde{\varphi}(\Gamma(\xi_1) = \Gamma(\xi_2)$ one has to require $\varphi$ to be open.

\begin{ex} \label{E:notonto}

   Let
   \[
    \E_1 := [0,1]\times (\mathbb{K}\times \mathbb{K}),\  \E_2 := (\{0\}\times (\{0\}\times \mathbb{K}))\cup ((0,1]\times (\mathbb{K}\times \mathbb{K}))                                                                                                         \mathbb{K}))
   \]
   and $p$ be the projection of $\E_1$ onto $[0,1]$. We give the linear space $\{t\}\times (\mathbb{K}\times \mathbb{K})\subset \E_1$ the norm
   $\|(t,(x,y))\| := \mathrm{max}(|x|,|y|)$. Then $\xi_1 := (\E_1,p,[0,1])$ is a trivial Banach bundle and $\xi_2 := (\E_2,p\mid\E_2,[0,1])$ a
   Banach subbundle of $\xi_1$. Define the Banach bundle map $\varphi : \E_1\to \E_2$ by $\varphi(t,(x,y)) := (t,(tx,y))$. Then $\varphi(\E_1) =
  \E_2$. It is obvious that the section $f(t) := (t,(t^{1/2},1))$ of $\Gamma(\xi_2)$ has no preimage in $\Gamma(\xi_1)$.

\end{ex}

\begin{thm} \label{T:sections}

   Let $\xi_i := (\E_i,p_1,T)$, $i = 1,2$, be Banach bundles with $T$ paracompact and $\varphi := \E_1\to \E_2$ an open Banach bundle map onto $\E_2$. Then
   $\tilde{\varphi}(\Gamma(\xi_1)) = \Gamma(\xi_2)$.

\end{thm}

\begin{proof}

Let $f\in \Gamma(\xi_2)$ and define $\Phi(t) := \varphi^{-1}(f(t))$, $t\in T$. We are going to show that $\Phi$ is lower semi-continuous. To
this end let $U\subset \E_1$ be open and suppose $U\cap \Phi(t_0)\neq \mset$. Thus there exists $x_0\in U$ with $p(x_0) = t_0$ and $\varphi(x_0)
= f(t_0)$. The set $\varphi(U)$ is open in $\E_2$ therefore there exists an open neighbourhood $V$ of $t_0$ such that $f(V)\subset \varphi(U)$.
It follows that $t_0\in V\subset \{t\in T\mid U\cap \Phi(t)\neq \mset\}$ and we conclude that the latter set is open. By Theorem
\ref{T:selection} $\Phi$ admits a continuous selection $g$. This is a section of $\xi_1$ that satisfies $\tilde{\varphi}(g) = f$.

\end{proof}

Thus it seems to be of some interest to establish conditions for a Banach bundle map which is onto the bundle space of the image to be open.

\begin{prop} \label{P:open}

   Let $\xi_1 := (\E_1,p_1,T)$ be a full Banach bundle and $\varphi$ a Banach bundle map of $\xi_1$ onto $\xi_2 := (\E_2,p_2,T)$. Then $\varphi$
   is open if and only if there exists an open cover $\{V_{\alpha}\}_{\alpha\in \A}$ of $T$ and positive numbers $\{m_{\alpha}\}_{\alpha\in \A}$
   such that for every $y\in p_2^{-1}(V_{\alpha})$ there is $x\in \E_1$ that satisfies $\varphi(x) = y$ and $\|x\|\leq
   m_{\alpha}\|y\|$, $\alpha\in \A$,

\end{prop}

\begin{proof}

   Suppose that $\varphi$ satisfies the condition; we are going to show that $\varphi$ is open on $p_1^{-1}(V_{\alpha})$, $\alpha\in \A$, and this
   will prove the 'if' part of the statement. So let $U\subset p_1^{-1}(V_{\alpha})$ be an open set and $y_0\in \varphi(U)$. Denote $t_0 :=
   p_2(y_0)\in V_{\alpha}$ and let $x_0\in U$ be such that $\varphi(x_0) = y_0$ and $f\in \Gamma(\xi_1)$ with $f(t_0) = x_0$. There exist an open
   neighbourhood $V$ of $t_0$ in $V_{\alpha}$ and $\varepsilon >0$ such that
   \[
    x_0\in \{x\in p_1^{-1}(V_{\alpha})\mid p_1(x)\in V,\ \|x - f(p_1(x))\| < \varepsilon\}\subset U.
   \]
   If not then there would exist a net $\{x_{\iota}\}$ in $p_1^{-1}(V_{\alpha})\setminus U$ such that $\{p_1(x_{\iota})\}$ converges to $t_0$ and
   $\|x_{\iota} - f(p_1(x_{\iota}))\|\to 0$. But that means $x_{\iota}\to f(t_0) = x_0\in U$, a contradiction. Hence $V$ and $\varepsilon > 0$
   as claimed indeed exist. Put $\tilde{f} := \tilde{\varphi}(f)$. We claim now that the open neighbourhood
   \[
    U^{\pr} ;= \{y\in p_2^{-1}(V_{\alpha})\mid p_2(y)\in V,\ \|y - \tilde{f}(p_2(y))\| < \varepsilon/m_{\alpha}\}
   \]
   of $y_0$ is contained in $\varphi(U)$. Indeed, let $y\in U^{\pr}$. By our assumption, there exists $x\in p_1^{-1}(p_2(y))$ such that
   $\varphi(x) = y - \tilde{f}(p_2(y))$ and
   \[
    \|x\|\leq m_{\alpha}\|y - \tilde{f}(p_2(y))\| < \varepsilon.
   \]
   Since
   \[
    \|(x + f(p_2(y))) - f(p_2(y))\| = \|x\| < \varepsilon
   \]
   we get $x + f(p_2(y))\in U$ and $y = \varphi(x + f(p_2(y)))\in \varphi(U)$ as needed.

   Suppose now that $\varphi$ is open. Let $t\in T$. The set $\varphi(\{x\in \E_1\mid \|x\| < 1\})$ is an open neighbourhood of $0_t$. There exist
   an open neighbourhood $V_t$ of $t$ and a positive number $k_t$ such that
   \[
    \{y\in \E_2\mid p_2(y)\in V_t,\ \|y\| < k_t\}\subset \varphi(\{x\in \E_1\mid \|x\| < 1\}).
   \]
   If not then there would exist a net $\{y_\iota\}$ in $\E_2\setminus \varphi(\{x\in \E_1\mid \|x\|\})$ such that $\{p_2(y_{\iota})\}$
   converges to $t$ and $\|y_{\iota}\|\to 0$. But then $y_{\iota}\to 0_t$, a contradiction. We are going to show that $V_t$ and $k_t$ have the property
   that for $y\in p_2^{-1}(V_t)$ there exists $x$ such that $\varphi(x) = y$ and
   \[
   \|x\|\leq \frac{2}{k_t}\|y\|.
   \]
   Now, for $y\in p_2^{-1}(V_t)$, $y\neq 0$, there exists $x^{\pr}\in p_1^{-1}(p_2(y))$ such that $\|x^{\pr}\| < 1$ and
   \[
    \varphi(x^{\pr}) = \frac{m_t}{2\|y\|}y.
   \]
   We got
   \[
    \varphi(\frac{2\|y\|}{k_t}x^{\pr}) = y,\ \ \|\frac{2\|y\|}{k_t}x^{\pr}\| \leq \frac{2}{k_t}\|y\|.
   \]

\end{proof}

The openness of a map between two Banach bundles intervenes also when looking at the preimage of a Banach subbundle. Returning to Example
\ref{E:notonto}, one can see that the preimage of the null subbundle $\{0_t\mid t\in T\}$ of $\xi_2$ by the map $\varphi$ is $\E^{\pr} :=
\{(0,(\mathbb{K}\times \{0\}))\}\cup \{(0,1]\times \{(0,0)\}$ which cannot be the space bundle of a subbundle of $\xi_1$. The subset
$\{(0,(\{\mathcal{R}e(x) > 0\mid x\in \mathbb{K}\}\times \{0\})\}$ is relatively open in $\E^{\pr}$ but its image by the projection $p_1$ is
$\{0\}$.

\begin{prop} \label{P:preim}

   Let $\varphi$ be an open Banach bundle map of the Banach bundle $\xi_1 := (\E_1,p_1,T)$ onto the Banach bundle $\xi_2 := (\E_2,p_2,T)$ and
   let \linebreak $\xi_2^{\pr} := (\E_2^{\pr},p_2\mid \E_2^{\pr},T)$ be a Banach subbundle of $\xi_2$. Then, with $\E_1^{\pr} := \varphi^{-1}(\E_2^{\pr})$,
   $(\E_1^{\pr},p_1\mid \E_1^{\pr},T)$ is a subbundle of $\xi_1$.

\end{prop}

\begin{proof}

   Let $U\subset \E_1$ be open and $U\cap \E_1^{\pr}\neq \mset$. Then $\varphi(U\cap E_1^{\pr}) = \varphi(U)\cap \E_2^{\pr}$ is relatively open
   in $\E_2^{\pr}$. It follows that $p_1(U\cap \E_1^{\pr}) = p_2(\varphi(U)\cap \E_2^{\pr})$ is open in $T$. We conclude that $p_1\mid
   \E_1^{\pr}$ is an open map for the relative topology of $\E_1^{\pr}$.

\end{proof}

We turn our attention now to the spaces of bounded sections. In order to have $\tilde{\varphi}(\Gamma_b(\xi_1))\subset \Gamma_b(\xi_2)$ for a
Banach bundle map between two Banach bundles it is quite natural to require $\sup_{t\in T}\|\varphi_t\| < \infty$. However, even in the situation of
Theorem \ref{T:sections} one does not necessarily have $\tilde{\varphi}(\Gamma_b(\xi_1)) = \Gamma_b(\xi_2)$. Indeed, let $\E_1 = \E_2 :=
\mathbb{N}\times \mathbb{K}$, $p_1 = p_2$ the projection on $\mathbb{N}$, and $\xi_i := (\E_i,p_i,\mathbb{N})$, $i = 1,2$, $\varphi(n,x) := (n,
x/n^2)$. The bounded section of $\xi_2$ given by $n\to (n,1/n)$ is the image of only one section of $\xi_1$: the unbounded section $n\to
(n,n)$.

\begin{thm} \label{T:bsections}

   Let $T$ be a paracompact space and $\xi_i := (\E_i,p_i,T)$, $i = 1,2$, Banach bundles. Let $\varphi$ be a Banach bundle map of $\xi_1$ onto
   $\xi_2$ such that $\sup_{t\in T}\|\varphi_t\| < \infty$. Then $\tilde{\varphi}(\Gamma_b(\xi_1)) = \Gamma_b(\xi_2)$ if and only if there is a
   positive number $m$ with the property that for each $y\in \E_2$ there exists $x\in \E_1$ such that $\varphi(x) = y$ and $\|x\|\leq m\|y\|$.

\end{thm}

\begin{proof}

   Suppose a constant $m > 0$ as in the above statement exists. Then $\varphi$ is an open map by Proposition \ref{P:open}. Let $f\in
   \Gamma_b(\xi_2)$ and define the set valued map $\Phi(t) := \{x\in p_1^{-1}(t)\mid \varphi(x) = f(t), \ \|x\| < (m + 1)\|f\|\}$, $t\in T$.
   Then $\Phi(t)$ is a non-void convex subset of $p_1^{-1}(t)$ for every $t\in T$. We are going to show that $\Phi$ is lower semi-continuous.
   Let $U$ be an open subset of $\E_1$ and suppose $U\cap \Phi(t_0)\neq \mset$. Thus there exists $x_0\in U\cap \{x\in \E_1\mid \|x\| < (m +
   1)\|f\|\}\cap \Phi(t_0)$. The subset $\varphi(U\cap \{x\in \E_1\mid \|x\| < (m + 1)\|f\|\}$ of $\E_2$ is an open neighbourhood of $f(t_0)$;
   hence there exists a neighbourhood $V$ of $t_0$ in $T$ such that
   \[
    f(t)\in \varphi(U\cap \{x\in \E_1\mid \|x\| < (m + 1)\|f\|\}). \ t\in V.
   \]
   Therefore, for each $t\in T$ there exists $x_t\in U\cap \{x\in \E_1\mid \|x\| < (m +1)\|f\|\}$ such that $\varphi(x_t) = f(t)$. It follows
   that $V$ is a neighbourhood of $t_0$ such that $U\cap \Phi(t)\neq \mset$ for every $t\in V$ and we obtained that $\Phi$ is lower
   semi-continuous. Now, it is easily seen that $t\to \overline{\Phi(t)}$ is lower semi-continuous too. Moreover, from $\Phi(t)\subset
   p_1^{-1}(t)$ and the fact that $p_1^{-1}(t)$ is closed in $\E_1$ we gather that $\|x\|\leq (m + 1)\|f\|$ for each $x\in \Phi(t)$. Thus, a
   continuous selection $g$ of $t\to \overline{\Phi(t)}$ given by Theorem \ref{T:selection} is a bounded section of $\xi_1$ that satisfies
   $\tilde{\varphi}(g) = f$ and we are done with this half of the proof.

   Suppose now that $\tilde{\varphi}(\Gamma_b(\xi_1) = \Gamma_b(\xi_2)$. It is a well known consequence of the Banach open mapping theorem that
   there exists a constant $k > 0$ with the property that for every $f\in \Gamma(\xi_2)$ there exists $g\in \Gamma_b(\xi_1)$ such that $\tilde{\varphi}(g) =
   f$ and $\|g\|\leq k\|f\|$. We claim that for every $y\in \E_2$ there exists $x\in \E_1$ such that $\varphi(x) = y$ and $\|x\|\leq 3k/2\|y\|$.
   Indeed, for $y\in E_2$, $y\neq 0$, denote $t := p_2(y)$ and let $f^{\pr}\in \Gamma_b(\xi_2)$ satisfy $f^{\pr}(t) = y$. There exists a
   neighbourhood $V$ of $t$ in $T$ such that $\|f^{\pr}(t^{\pr})\| < 3/2\|f^{\pr}(t)\|$ whenever $t^{\pr}\in V$. Let $h : T\to [0,1]$ be a
   continuous function such that $h(t) = 1$ and $h\mid (T\setminus V)\equiv 0$. Then $f ;= hf^{\pr}$ satisfies $f(t) = y$ and $\|f\|\leq
   3/2\|y\|$. Let now $g\in \Gamma_b(\xi_1)$ be such that $\tilde{\varphi}(g) = f$ and $\|g\|\leq k\|f\|$. Then $\varphi(g(t)) = y$ and
   \[
    \|g(t)\|\leq \|g\|\leq k\|f\|\leq 3k/2\|y\|
   \]
   and this establishes our claim.

\end{proof}

\begin{rem}

   The 'if' direction of the above theorem in the case of a quotient map is part of \cite[Theorem 9.14]{G}.

\end{rem}

For a Banach bundle $\xi$ whose base space is locally compact Hausdorff the space $\Gamma_0(\xi)$ is of interest and we shall look now at the
behavior of such spaces under a Banach bundle map.

\begin{lem} \label{L:openset}

   Let $\xi := (\E,p,T)$ be a Banach bundle and $h : T\to (0,\infty)$ a continuous function. Then $\F := \{x\in \E\mid p(x) = t,\ \|x\| < h(t)\}$ is
   an open subset of $\E$.

\end{lem}

\begin{proof}

   Let $x_0\in \F$ and put $t_0 = p(x_0)$. There exists an open neighbourhood $V$ of $t_0$ such that $h(t) > h(t_0) - 1/2(h(t_0) - \|x_0\|)$ if
   $t\in V$. Now, it is easily checked that $p^{-1}(V)\cap \{x\in \E\mid \|x\| < 1/2(\|x_0\| + h(t_0))\}\subset \F$ and we are done.

\end{proof}

\begin{thm} \label{T:Gamma0}

   Let $\xi_1 := (\E_1,p_1,T)$ be a Banach bundle and $\xi_2 := (\E_2,p_2,T)$ be a continuous Banach bundle with $T$ locally compact Hausdorff.
   Suppose that $\varphi$ is a Banach bundle map of $\E_1$ onto $\E_2$ such that $\sup_{t\in T} \|\varphi_t\| < \infty$. If there exists a
   positive constant $m$ with the property that for each $y\in \E_2$ there exists $x\in \E_1$ satisfying $\varphi(x) = y$ and $\|x\| < m\|y\|$
   then $\tilde{\varphi}(\Gamma_0(\xi_1)) = \Gamma_0(\xi_2)$.

\end{thm}

\begin{proof}

   Let $g\in \Gamma_0(\xi_1)$ and $\varepsilon > 0$. From
   \[
    \|\tilde{\varphi}(g)(t)\| = \|\varphi(g(t))\|\leq \|g(t)\|\sup_{t^{\pr}\in T} \|\varphi_{t^{\pr}}\|
   \]
   we obtain that $\{t\in T\mid \|\tilde{\varphi}(g)(t)\|\geq \varepsilon\}$ is a closed subset of the compact set \linebreak
   $\{t\in T\mid \|g(t)\|\geq
   \varepsilon/\sup_{t^{\pr}}\|\varphi_{t^{\pr}}\|\}$. Thus $\tilde{\varphi}(g)\in \Gamma_0(\xi_2)$ and we got
   $\tilde{\varphi}(\Gamma_0(\xi_1))\subset \Gamma_0(\xi_2)$.

   Let now $f\in \Gamma_0(\xi_2)$. Then $T^{\pr} := \{t\in T\mid \|f(t)\| > 0\}$ is an open $\sigma$-compact subset of $T$; it is paracompact by
   \cite[Corollary 2, p. 211]{E}. Define
   \[
    \Phi(t) := \{x\in p_1^{-1}(t)\mid \varphi(x) = f(t),\ \|x\| < (m + 1)\|f(t)\|\}
   \]
   for $t\in T^{\pr}$. Then $\Phi(t)$ is a non-void convex subset of the Banach space $p_1^{-1}(t)$. We are going to show that $t\to \Phi(t)$ is
   lower semi-continuous on $T^{\pr}$. With $U$ an open subset of $p_1^{-1}(T^{\pr})$ suppose $U\cap \Phi(t_0)\neq \mset$, $t_0\in T^{\pr}$.
   Let $x_0\in U\cap
   \Phi(t_0)$. Then $\|x_0\| < (m + 1)\|f(t_0)\|$ and $x_0\in U\cap \{x\in p_1^{-1}(T^{\pr})\mid \|x\| < 1/2(\|x_0\| + (m + 1)\|f(t_0)\|)\}$.
   The set $U^{\pr} := U\cap \{x\in p_1^{-1}(T^{\pr})\mid \|x\| < 1/2(\|x_0\| + (m + 1)\|f(t_0)\|)\}$
   is an open subset of $\E_1$ and $\varphi$ is an open map by Proposition \ref{P:open}; hence $\varphi(U^{\pr})$ is an open
   neighbourhood of $\varphi(x_0) = f(t_0)$. There exists an open neighbourhood $V$ of $t_0$ in $T^{\pr}$ such that $f(t)\in \varphi(U^{\pr})$
   and $\|f(t)\| > \|f(t_o)\| - 1/2(\|f(t_0)\| - \|x_0\|/(m+1))$ for $t\in V$. Thus, if $t\in V$ there exists $x_t\in U^{\pr}$ such that
   $\varphi(x_t) = f(t)$. Then $\|x_t\| < 1/2(\|x_o\| + (m + 1)\|f(t_0)\|) < (m + 1)\|f(t)\|$ and we have $x_t\in U\cap \Phi(t)$. We got $t_0\in
   V\subset U\cap \Phi(t)$ and we can conclude that $t\to \Phi(t)$ is indeed lower semi-continuous.

   The set valued map $t\to \overline{\Phi(t)}$, $t\in T^{\pr}$, is lower semi-continuous too and each set $\overline{\Phi(t)}$ is a non-void closed
   subset of $\{x\in p_1^{-1}(t)\mid \|x\|\leq (m + 1)\|f(t)\|\}$. By Theorem \ref{T:selection} there exists a continuous selection $g^{\pr}$ of
   $\Phi$ on $T^{\pr}$. We have $\varphi(g^{\pr}(t)) = f(t)$ and $\|g^{\pr}(t)\|\leq (m + 1)\|f(t)\|$ for each $t\in T^{\pr}$. Define now
   \[
    g(t) =
    \begin{cases}
       g^{\pr}(t), \ &\text{if $t\in T^{\pr}$},\\
       0_t, \ &\text{if $t\in T\setminus T^{\pr}$}.
    \end{cases}
   \]
   The function $g : T\to \E$ is continuous. Indeed, let $\{t_{\alpha}\}$ be a net in $T^{\pr}$ that converges to $t\in T\setminus T^{\pr}$.
   then $f(t_{\alpha})\to f(t) = 0_t$; from $0\leq \|g(t_{\alpha})\|\leq (m + 1)\|f(t_{\alpha})\|$ we derive $\|g(t_{\alpha})\|\to 0$. Hence
   $g(t_{\alpha})\to 0_t = g(t)$. If $\varepsilon > 0$ then $\{t\in T\mid \|g(t)\|\geq \varepsilon\}$ is compact since it is a closed subset of the
   compact set $\{t\in T\mid \|f(t)\|\geq \varepsilon/(m+1)\}$. We got $g\in \Gamma_0(\xi_1)$ with $\tilde{\varphi}(g) = f$ and we conclude
   $\tilde{\varphi}(\Gamma_0(\xi_1)) = \Gamma_0(\xi_2)$.

   \end{proof}

   We shall now consider results suggested by \cite{BG} and \cite{M} on the existence of right inverses for some Banach bundle maps.

   \begin{prop} \label{P:BG1}

      Let $\varphi$ be an open Banach bundle map of $\xi_1 := (\E_1,p_1,T)$ onto $\xi_2 := (\E_2,p_2,T)$. Suppose that $\E_2$ is paracompact.
      Then there exists a continuous map $\psi : \E_2\to \E_1$ such that $\varphi(\psi(y)) = y$ for every $y\in \E_2$.

   \end{prop}

   \begin{proof}

      Define $\Phi(y) := \varphi^{-1}(y)$, $y\in \E_2$. By using the fact that $\varphi$ is open one easily sees that $\Phi$ is lower
      semi-continuous. An application of Theorem \ref{T:selection} in the obvious manner yields a continuous selection $\psi : \E_2\to \E_1$ of
      $\Phi$ that fulfills what is needed.

   \end{proof}

   To obtain a homogeneous right inverse we shall suppose that the range of the Banach bundle map is a continuous Banach bundle. But first we
   have to state two simple lemmas. We omit their elementary proofs; the first can be proven by the means of an obvious compactness argument and
   the second lemma is an easy consequence of the first. We let $\mu$ be the normalized Lebesgue measure on the unit circle in $\mathbb{C}$ and
   $\widehat{\lambda_1\lambda_2}$ will denote the arc of the unit circle from $\lambda_1$ to $\lambda_2$ in the positive direction.

   \begin{lem} \label{L:first}

   Let $\xi_i := (\E_i,p_i,T)$, $i = 1,2$ be two Banach bundles over the complex field,  $f : \E_2\to \E_1$ a continuous function and $\{y_{\alpha}\}_{\alpha\in \A}$
   a net in $\E_2$ that converges to $y\in \E_2$. Given $\varepsilon > 0$ there exists $\delta > 0$ such that $|f(\lambda_1 y_{\alpha}) -
   f(\lambda_2 y_{\alpha})| < \varepsilon$ and $|f(\lambda_1 y) - f(\lambda_2 y)| < \varepsilon$ whenever $|\lambda_1| = |\lambda_2| = 1$,
   $\mu(\widehat{\lambda_1\lambda_2}) < \delta$ and for every $\alpha\in \A$.

   \end{lem}

   \begin{lem} \label{L:second}

   Let $\xi_1$, $\xi_2$, $f$, $\{y_{\alpha}\}_{\alpha\in \A}$, and $y$ be as in Lemma \ref{L:first}. Given $\varepsilon > 0$ there exists
   $\delta > 0$ with the following property: if $\{\lambda_i\}_{i=1}^n$, $\lambda_1 = \lambda_n$, is a division of the unit circle in the
   complex plane with mesh smaller than $\delta$ then
   \[
    \left|\underset{|\lambda|=1}{\int} f(\lambda y)d\mu(\lambda) - \sum_{i=1}^{n-1} f(\lambda_i y)\mu(\widehat{\lambda_i \lambda_{i+1}})\right| < \varepsilon
   \]
   and
   \[
    \left|\underset{|\lambda|=1}{\int} f(\lambda y_{\alpha})d\mu(\lambda) - \sum_{i=1}^{n-1}
                                                                 f(\lambda_i y_{\alpha})\mu(\widehat{\lambda_i \lambda_{i+1}})\right| < \varepsilon
   \]
   for every $\alpha\in \A$.

   \end{lem}

   \begin{prop} \label{P:hom}

      Let $\varphi$ be a Banach bundle map from the Banach bundle $\xi_1 := (\E_1,p_1,T)$ onto the continuous Banach bundle $\xi_2 :=
      (\E_2,p_2,T)$. Suppose that $\E_2$ is paracompact and there exists $m > 0$ with the propriety that for every $y\in \E_2$
      there exists $x\in \E_1$ such that $\varphi(x) = y$ and
      $\|x\|\leq m\|y\|$. Then there exists a continuous map $\psi := \E_2\to \E_1$ such that $\varphi(\psi(y)) = y$ and $\psi(\lambda y) =
      \lambda\psi(y)$ for every $y\in \E_2$ and $\lambda\in \mathbb{K}$.

   \end{prop}

   \begin{proof}

      The set $S := \{y\in \E_2\mid \|y\| = 1\}$ is closed in $\E_2$ thus paracompact. We define $\Phi(y) := \{x\in \E_1\mid \varphi(x) = y,
      \|x\| < m + 1\}$, $y\in S$. Then $\Phi(y)$ is a non-empty convex subset of $p_1^{-1}(p_2(y))$ and we want to show that $\Phi$ is lower
      semi-continuous. So let $U$ be an open subset of $\E_1$ and suppose $U\cap \Phi(y_0)\neq \mset$. Then $U^{\pr} :=
      \varphi(U\cap \{x\in \E_1\mid \|x\| < m + 1\})$ is an open subset of $\E_2$ since $\varphi$ is an open map by Proposition \ref{P:open} and
      $y_0\in U^{\pr}\cap F$. If $y\in U^{\pr}\cap F$ then there exists $x\in U$ such that $\|x\| < m + 1$ and $\varphi(x) = y$. Thus $U\cap
      \Phi(y)\neq \mset$ if $y\in U^{\pr}\cap F$ and we have proved that $\Phi$ is lower semi-continuous. The map $y\to \overline{\Phi(y)}$ is
      also lower semi-continuous and Theorem \ref{T:selection} yields a continuous function $g : F\to \E_1$ such that $\varphi(g(y)) = y$ and
      $\|g(y)\|\leq m + 1$ for $y\in F$. By defining
      \[
       h(y) :=
       \begin{cases}
          \|y\|g(y/\|y\|), \ &\text{if $\|y\|\neq 0$},\\
          0, \ &\text{otherwise}
       \end{cases}
      \]
      for $y\in \E_2$ one gets a continuous function from $\E_2$ to $\E_1$ such that $\varphi(h(y)) = y$ for $y\in \E_2$..
      If the scalars are real then one easily checks that $\psi (y) := 1/2(h(y) - h(-y))$ has the required properties.
      Suppose now that the scalars are complex. Then we define
      \[
       \psi(y) := \underset{|\lambda|=1} {\int} \bar{\lambda}h(\lambda y)d\mu(\lambda),\quad y\in \E_2.
      \]
      The function $\psi$ has all the needed properties. We shall prove only its continuity all its other attributes being straightforward to
      check. So let $y\in \E_2$ and suppose $\{y_{\alpha}\}_{\alpha\in \A}$ is a net that converges to $y$. We denote $t := p_2(y)$ and
      $t_{\alpha} := p_2(y_{\alpha})$. Let $U$ be an open neighbourhood of $\psi(y)$. Now $T$ is paracompact being
      homeomorphic to the closed subspace $\{0_{\tau}\}_{\tau\in T}$ of $\E_2$. So there exist a section $\sigma$ of $\xi_1$, a neighbourhood
      $V$ of $t$ contained in $p_1(U)$ and $\varepsilon > 0$ such that $\sigma(t) = \psi(y)$ and
      \[
       U^{\pr} := \{x\in \E_1\mid p_1(x)\in V, \quad \|x - \sigma(p_1(x))\| < 2\varepsilon\}\subset U.
      \]
      Let now $\delta$ be the positive number given by
      Lemma \ref{L:second} for our number $\varepsilon$ and the integrant appearing in the definition of $\psi$. Suppose $\{\lambda_i\}_{i=1}^n$,
      $\lambda_1 = \lambda_n$, is a division of the unit circle of mesh less than $\delta$. Then $\|\psi(y) - \sum_{i=1}^{n-1}
      \bar{\lambda}_ih(\lambda_i y)\mu(\widehat{\lambda_1 \lambda_{i+1}})\| < \varepsilon$. We denote
      \[
       U^{\pr \pr} := \{x\in \E_1\mid p_1(x)\in V,\quad \|x - \sigma(p_1(x))\| < \varepsilon\};
      \]
      then $\sum_{i=1}^{n-1} \bar{\lambda}_ih(\lambda_i
      y)\mu(\widehat{\lambda_i \lambda_{i+1}})$ belongs to the open set $U^{\pr \pr}$. There exists $\alpha_0\in \A$ such that if
      $\alpha\succ \alpha_0$ then $\sum_{i=1}^{n-1} \bar{\lambda}_ih(\lambda_i y_{\alpha})\mu(\widehat{\lambda_i \lambda_{i+1}})$ also belongs to
      $U^{\pr \pr}$. Thus we have $\|\sum_{i=1}^{n-1} \bar{\lambda}_ih(\lambda_i y_{\alpha})\mu(\widehat{\lambda_i \lambda_{i+1}}) -
      \sigma(p_2(y_{\alpha}))\| < \varepsilon$ if $\alpha \succ \alpha_0$ and
      \[
       \|\psi(y_{\alpha}) - \sum_{i=1}^{n-1} \bar{\lambda}_ih(\lambda_i y_{\alpha})\mu(\widehat{\lambda_i \lambda_{i+1}})\| < \varepsilon.
      \]
      We obtained that $\psi(y_{\alpha})\in U^{\pr}\subset U$ if $\alpha$
      is large enough so we conclude that $\{\psi(y_{\alpha})\}$ converges to $\psi(y)$.

   \end{proof}

   \begin{rem}

      It follows from the proof of Proposition \ref{P:hom} that the map $\psi$ satisfies $\|\psi(y)\|\leq (m + 1)\|y\|$ for every $y\in \E_2$. Thus one
      can easily obtain a proof of Theorem \ref{T:Gamma0} by using Proposition \ref{P:hom} if in addition to the other hypotheses of that
      theorem one supposes that $\E_2$ is paracompact.

   \end{rem}

   \section{M-ideals of $\Gamma_0(\xi)$} \label{S:M-ideals}

   For the present section $\xi := (\E,p,T)$ will be a fixed Banach bundle. We shall mainly discuss the M-ideals
   of the Banach space $\Gamma_0(\xi)$ when the base space is locally compact Hausdorff; most of the results that follow are generalizations of results from \cite{G} where only the case of a
   compact Hausdorff base space was considered. As in \cite{G}, for a closed subset $A$ of $T$ the notation $N_A$ will stand for $\{\sigma\in
   \Gamma_b \mid \sigma(t) = 0_t \quad \text{for all} \quad t\in A\}$ or for $\{\sigma\in \Gamma_0(\xi) \mid \sigma(t) = 0_t \quad \text{for all} \quad t\in
   A\}$ if $T$ is locally compact Hausdorff; it will be clear from the context which is the case. If $A$ is a singleton, say $\{t\}$, then we
   shall use the notation $N_t$ instead of $N_{\{t\}}$.

   In \cite[Proposition 13.6]{G} the base space is assumed to be compact Hausdorff. However the proof given there is valid
   with obvious modifications in the more general situations detailed in the next Proposition. Of course, we omit its proof.

   \begin{prop} \label{P:M-id}

      The closed subspace $N_A$ is an M-ideal of $\Gamma_b(\xi)$ if $A$ is a closed subset of the normal space $T$.
      For a locally compact Hausdorff base space $T$ $N_A$ is an M-deal of $\Gamma_0(\xi)$ if $A$ is a compact
      subset of $T$ or if $A$ is a closed subset of $T$ such that the complement of its interior is compact.

   \end{prop}

   \begin{prop} \label{P:null}

      Suppose $T$ is paracompact (locally compact Hausdorff) and $A$ and $B$ are closed (compact, respectively) subsets of $T$. Then $N_{A\cap
      B}  = N_A + N_B$.

   \end{prop}

   \begin{proof}

      We must prove only $N_{A\cap B}\subset N_A + N_B$. Let $\sigma\in N_{A\cap B}$ and define $\rho : A\cup B\to \E$ by $\rho(t) := \sigma(t)$
      if $t\in A$ and $\rho(t) := 0_t$ if $t\in B$. Then $\rho$ is well defined and continuous on $A\cup B$. Corollary \ref{C:ext} provides us
      with an extension of $\varrho$ in $\Gamma_b(\xi)$ ($\Gamma_0(\xi)$, respectively); we shall call this extension $\rho_B$. We have $\rho_B\in
      N_B$ and $\rho_A := \sigma - \rho_B$ is in $A$.

   \end{proof}

   Proposition \ref{P:null} appeared as \cite[Corollary 15.8]{G} for $T$ compact Hausdorff and was given there a different proof from the one above.

   An M-ideal $L$ of a Banach space $X$ is called primitive if there is an extreme point $\omega$ of the closed unit ball of the dual space
   $X^{\pr}$ such that $L$ is maximal among the M-ideals contained in the null subspace of $\omega$. If $L$ is a primitive M-ideal and $L_1$,
   $L_ 2$ are M-ideals such that $L_1\cap L_2\subset L$ then $L_1\subset L$ or $L_2\subset L$. Every M-ideal is the intersection of the primitive
   M-ideals containing it. For all these facts see \cite[Section 3]{AE}.

   In the remainder of this section the base space $T$ will always be a locally compact Hausdorff space.

   The case $T$ compact Hausdorff of the following proposition is part of \cite[Proposition 13.11]{G}.

   \begin{prop} \label{P:prim}

      If $L$ is a primitive M-ideal of $\Gamma_0(\xi)$ then there is a unique $t_0\in T$ such that $N_{t_0}\subset L$.

   \end{prop}

   \begin{proof}

      Let $K$ be a non-void compact subset of $T$. Then $N_K$ and $N_{\overline{T\setminus K}}$ are M-ideals of $\Gamma_0(\xi)$ by Proposition
      \ref{P:M-id} and $N_K\cap N_{\overline{T\setminus K}} =  \{0\}\subset L$. We infer that $N_K$ or $N_{\overline{T\subset K}}$ is included
      in $L$. If for every such $K$ we have $N_{\overline{T\setminus K}}\subset L$ then $L = \Gamma_0(\xi)$ since the sections with compact
      support are dense in $\Gamma_0(\xi)$. Thus there exists a non-void compact set $K$ such that $N_K\subset L$. Let $\F$ be the non-void family of all
      non-void compact subsets of $T$ such that $N_K\subset L$. If $K_1$ and $K_2$ belong to $\F$ then $N_{K_1\cap K_2} = N_{K_1} +
      N_{K_2}\subset L$ and it follows that $\F$ has the finite intersection property. Hence $K_0 := \cap\{K\mid K\in \F\}\neq \mset$.

      Denote
      \[
       \Lambda := \overline{\underset{K\in \F}{\cup} N_K}.
      \]
      We claim that $N_{K_0} = \Lambda$. Obviously $\Lambda\subset N_{K_0}$. Let now $\sigma\in N_{K_0}$ and $\varepsilon > 0$. Set $V := \{t\in
      T\mid \|\sigma(t)\| < \varepsilon\}$; then $V$ is open and $K_0\subset V$. If it were true that for each $K\in \F$ there would be $t_K\in
      K\setminus V$ then the net $\{t_K\}_{K\in \F}$ would have a convergent subnet whose limit would be in $K_0$. On the other hand,
      $\|\sigma(t_k)\| \geq \varepsilon$ for all $K\in \F$, a contradiction. Thus there exists $K\in \F$ such that $K\subset V$. Let $f :T\to
      [0,1]$ be a continuous function such that $f\mid K\equiv 0$ and $f\mid (T\setminus V)\equiv 1$. Then with $\sigma^{\pr} := f\cdot \sigma$ we
      have $\sigma^{\pr}\in N_K$ and $\|\sigma - \sigma^{\pr}\| < \varepsilon$. We infer that $\sigma\in \Lambda$ and the claim follows
      from this. In particular, $K_0\in \F$.

      It remains only to show that $K_0$ consists of only one point. If there are in $K_0$ at least two distinct points then there exist compact
      subsets $K_1,K_2$ of $K_0$ such that $K_0 = K_ 1\cup K_2$ and $K_1\neq K_0\neq K_2$. These subsets will have to satisfy $N_{K_1}\cap
      N_{K_2} = N_{K_0}\subset L$. Hence $N_{K_1}\subset L$ or $N_{K_2}\subset L$. But $K_i\in \F$, $i = 1\  \text{or}\  2$, and $K_i\subsetneqq K_0$
      is a contradiction. Therefore $K_0 = \{t_0\}$ for some $t_0\in T$ and we found $N_{t_0}\subset L$.

      The uniqueness of $t_0$ with the property $N_{t_0}\subset L$ follows from the definitions of $\F$ and $K_0$.

   \end{proof}

   Now we can apply the same proof as in the compact case of \cite[Proposition 15.20]{G} to get

   \begin{prop} \label{P:module}

      Every M-ideal $L$ of $\Gamma_0(\xi)$ is a $C_b(T)$-module.

   \end{prop}

   \begin{proof}

      Suppose first that $L$ is a primitive M-ideal and let $t_0\in T$ be the element given by Proposition \ref{P:prim}. If $f\in C_b(T)$ and
      $\sigma\in L$ then $f\cdot \sigma - f(t_0)\sigma$ belongs to $N_{t_0}\subset L$. Hence $f\cdot \sigma\in L$ and the case of a primitive
      M-ideal is settled. The general case follows from this particular case and the fact mentioned above that each M-ideal is the intersection
      of the primitive M-ideals containing it.

   \end{proof}

   We proceed now to give a characterization of those closed subspaces of $\Gamma_0(\xi)$ that are $C_b(T)$-modules. The parallel result for paracompact
   base spaces and spaces of bounded sections is \cite[Theorem 8.6]{G}. First we state and prove a simple lemma.

   \begin{lem} \label{L:approx}

      Let $L$ be an M-ideal of $\Gamma_0(\xi)$, $\sigma\in L$, and suppose that $\|\sigma(t_0)\| < \theta$ for some $t_0\in T$ and $\theta > 0$.
      Then there is $\sigma^{\pr}\in L$ such that $\sigma^{\pr}(t_0) = \sigma(t_0)$ and $\|\sigma^{\pr}\| < \theta$.

   \end{lem}

   \begin{proof}

      Denote $\theta^{\pr} := 1/2(\theta + \|\sigma(t_0)\|)$. Then $V := \{t\in T\mid \|\sigma(t)\| < \theta^{\pr}\}$ is an open neighbourhood of $t_0$ and
      $T\setminus V$ is compact. Let $f : T\to [0,1]$ be a continuous function such that $f(t_0) = 1$ and $f\mid (T\setminus V)\equiv 0$. Define
      $\sigma^{\pr} := f\cdot \sigma$. It is easily  checked that $\sigma^{\pr}$ has the required properties.

   \end{proof}

   \begin{thm} \label{T:subb}

      A closed subspace of $\Gamma_0(\xi)$ is a $C_b(T)$-submodule if and only if there exists a subbundle $\eta$ of $\xi$ such that $L =
      \Gamma_0(\eta)$.

   \end{thm}

   \begin{proof}

      Obviously $\Gamma_0(\eta)$ for $\eta$ a subbundle of $\xi$ is a $C_b(T)$-submodule of $\Gamma_0(\xi)$. Suppose now that the closed
      subspace $L$ of $\Gamma_0(\xi)$ is a $C_b(T)$-submodule. Denote $\F(t) := \{\sigma(t)\mid \sigma\in L\}$, $t\in T$. Clearly $\F(t)$ is a
      linear subspace of $\E(t)$ and we are going to show that it is closed in $\E(t)$. Let $\{x_n\}$ be a sequence in $\F(t)$ that converges to
      $x\in \E(t)$. By passing to a subsequence if necessary we may suppose that $\|x_{n+1} - x_n\| < 2^{-n}$ for each $n\in \mathbb{N}$. There
      exist $\sigma_n\in L$ such that $\sigma_n(t) = x_n$, $n\in \mathbb{N}$. Put $\varrho_1 : = \sigma_1$; Lemma \ref{L:approx} yields
      $\varrho_n\in L$ with $\varrho_n(t) = \sigma_{n+1}(t) - \sigma_n(t)$ and $\|\varrho_n\| < 2^{-n}$, $n\geq 2$. The section $\varrho ;=
      \sum_{n=1}^{\infty} \varrho_n$ is in $L$ and $\varrho(t) = \lim \sigma_n(t) = x$ and we deduce that $x\in \F(t)$.

      Set $\F := \cup_{t\in T} \F(t)$; we are going to show that $p\mid \F$ is open. Let $U$ be an open subset of $\E$ and $t_0\in p(U\cap \F)$.
      There exists $\sigma\in L$ such that $\sigma(t_0)\in U\cap \F$. there exists a neighbourhood $V$ of $t_0$ in $T$ such that $\sigma(t)\in U$
      whenever $t\in V$. Then $V\subset p(U\cap \F)$ and we obtained that $p(U\cap \F)$ is an open subset of $T$. We have shown that $\eta :=
      (\F,p\mid\F,T)$ is a subbundle of $\xi$.

      We want to show now that $L = \Gamma_0(\eta)$. Clearly $L\subset \Gamma_0(\eta)$. Let $\sigma\in \Gamma_0(\eta)$ and $\varepsilon > 0$.
      Then $K := \{t\in T\mid \|\sigma(t)\|\geq \varepsilon\}$ is a compact subset of $T$ and, $L$ being a $C_b(T)$ module, for each $t\in K$
      there exists $\varrho_t\in L$ such that $\varrho_t(t) = \sigma(t)$ and $\varrho_t\mid (T\setminus K) \equiv 0$. The set $V_t :=
      \{t^{\pr}\in T\mid \|\sigma(t^{\pr}) - \varrho_t(t^{\pr})\| < \varepsilon\}$, $t\in K$, is open; for each $t\in K$ we choose on open set
      $W_t$ such that $\overline{W_t}$ is compact and $t\in W_t\subset \overline{W_t}\subset V_t$. There is a finite subcover $\{W_{t_i}\}_{i=1}^n$
      of the cover $\{W_t\}_{t\in K}$ of $K$; let $f_i : T\to [0,1]$ be a continuous function such that $f_i\mid \overline{W_{t_i}} \equiv 1$
      and $f_i\mid (T\setminus V_{t_i}) \equiv 0$, $1\leq i\leq n$. On $\cup_{i=1}^n \overline{W_{t_i}}$ set
      \[
       g_i := \frac{f_i}{\sum_{j=1}^n f_j}
      \]
      and extend each $g_i$ to a continuous and bounded function on $T$. For simplicity, the extension will be denoted also by $g_i$. Define
      now $\rho := \sum_{i=1}^n g_i\varrho_{t_i}$. Then $\rho\in L$ and for $t\in K$ we have
      \[
       \|\rho(t) - \sigma(t)\|\leq \sum_{i=1}^n g_i\|\varrho_{t_i} - \sigma(t)\| < \varepsilon.
      \]
      Define the open neighbourhood $V := \{\|\rho(t) - \sigma(t)\| < \varepsilon$ of $K$ and let $h : T\to [0,1]$ be a continuous function such
      that $h\mid K\equiv 1$ and $h\mid (T\setminus V)\equiv 0$. Then $h\cdot \rho\in L$ and $\|h(t)\rho(t) - \sigma(t)\| < \varepsilon$ on $K$.
      If $t\in T\setminus V$ then $\|h(t)\rho(t) - \sigma(t)\| = \|\sigma(t)\| < \varepsilon$ since $K\subset V$. Now, if $t\in V\setminus K$ we
      have
      \begin{multline}
       \|h(t)\rho(t) - \sigma(t)\|\leq |h(t) - 1|\|\rho(t)\| + \|\rho(t) - \sigma(t)\|\leq\\
                                           (1 - h(t))(\|\rho(t) - \sigma(t)\| + \|\sigma(t)\|) + \|\rho(t) - \sigma(t)\| < 3\varepsilon.
      \end{multline}
      We have found $h\cdot \rho\in L$ such that $\|h\cdot \rho - \sigma||\leq 3\varepsilon$ hence $\sigma\in L$ and we conclude that $L =
      \Gamma_0(\eta)$.

      \end{proof}

      The case of a compact base space of the following result is \cite[Theorem 15.21]{G} for which only a brief indication of the proof was
      given.

      \begin{thm} \label{T:char}

         A closed subspace $L$ of $\Gamma_0(\xi)$ is an M-ideal if and only if there exists a
         subbundle $\eta := (\F,p\mid \F, t)$ of $\xi$ such that $L = \Gamma_0(\eta)$ and and each fiber $\F(t)$ is an M-deal of the Banach space
         $\E(t)$.

      \end{thm}

      \begin{proof}

         Suppose first that $L$ is an M-ideal of $\Gamma(\xi)$. It follows from Proposition \ref{P:module} and Theorem \ref{T:subb} that there
         exists a subbundle $\eta := (\F,p\mid \F,T)$ of $\xi$ such that $L = \Gamma_0(\eta)$. It remains to show that $\F(t) = \{\sigma(t)\mid
         \sigma\in L\}$ is an M-ideal of $\E(t)$. For this purpose we proceed to show that $\F(t_0)$ has the 3-ball-property in $\E(t_0)$, $t_0\in
         T$. Let $\{x_i\}_{i=1}^3\subset \E(t_0)$ and let $\{r_i\}_{i=1}^3$ be positive numbers such that there exists $x\in \E(t_0)$ with
         $\|x_i - x\| < r_i$, $1\leq i\leq 3$, and there exist $\{y_i\}_{i=1}^3$ in $\F(t_0)$ such that $\|x_i - y_i\| < r_i$, $1\leq i\leq 3$.
         There are sections $\sigma,\sigma_i\in \Gamma_0(\xi)$, such that $\sigma(t_0) = x$ and $\sigma_i(t_0) = x_i$, $1\leq i\leq 3$, and
         sections $\varrho_i\in \Gamma_0(\eta)$ such that $\varrho_i(t_0) = y_i$, $1\leq i\leq 3$. The set
         \[
          V ;= \{ t\in T\mid \|\sigma_i(t) - \sigma(t)\| < r_i,\ \|\sigma_i(t) - \varrho_i(t)\| < r_i,\ 1\leq i\leq 3\}
         \]
         is an open neighbourhood of $t_0$. Let $f : T\to [0,1]$ be a continuous function such that $f(t_0) = 1$ and $f\mid (T\setminus V)
         \equiv 0$. Then $f\cdot \sigma, f\cdot \sigma_i\in \Gamma_0(\xi)$ and $f\cdot \varrho_i\in \Gamma_0(\eta)$, $1\leq i\leq 3$. We have
         $\|f\cdot \sigma_i - f\cdot \sigma\| < r_i$ and $\|f\cdot \sigma_i - f\cdot \varrho_i\| < r_i$, $1\leq i\leq 3$. Since $\Gamma_0(\eta)$
         is an M-ideal there is a section $\varrho\in \Gamma_0(\eta)$ such that $\|f\cdot \sigma_i - \varrho\| < r_i$, $1\leq i\leq 3$. Then
         $\varrho(t_0)\in \F(t_0)$ and $\|x_i - \varrho(t_0)\| = \|f(t_0)\sigma_i(t_0) - \varrho_i(t_0)\| < r_i$, $1\leq i\leq 3$. Thus
         $\F(t_0)$ has indeed the 3-ball-property in $\E(t_0)$.

         Now we suppose that $\eta := (\F,p\mid \F,t)$ is a subbundle of $\xi$ such that each $\F(t)$ is an M-ideal of $\E(t)$. Since each
         subspace $\F(t)$ is closed in $\E(t)$ it is obvious that $\Gamma_0(\eta)$ is a closed subspace of $\Gamma_0(\xi)$. We are going to show
         that $\Gamma_0(\eta)$ has the 3-ball-property in $\Gamma_0(\xi)$ and is thus an M-ideal of $\Gamma_0\xi)$. Let $\{r_i\}_{i=1}^3$ be
         positive numbers, $\sigma_i,\sigma\in \Gamma_0(\xi)$, and $\varrho_i\in \Gamma_0(\eta)$, $1\leq i\leq 3$, so that $\|\sigma_i -
         \sigma\|< r_i$ and $\|\sigma_i - \varrho_i\| < r_i$, $1\leq i\leq 3$. Let $r_i^{\pr}$ be a number satisfying
         \[
          \max\{\|\sigma_i - \sigma\|,\ \|\sigma_i - \varrho_i\|\} < r_i^{\pr} < r_i,
         \]
         $1\leq i\leq 3$ and set
         \[
          C(t) := \cap_{i=1}^3 \{y\in \F(t)\mid \|\sigma_i(t) - y\| < r_i^{\pr}\}.
         \]
         The subspace $\F(t)$ of $\E(t)$ has the 3-ball-property; hence $C(t)$ is a non-void convex subset of $\F(t)$, $t\in T$. We are going to
         show that the map $t\to C(t)$ is lower semi-continuous. Let $U$ be an open subset of $\E$ and suppose $C(t_0)\cap U\neq \mset$ for some
         $ t_0\in T$. Then there exists an element $y_0\in C(t_0)\cap U$ and $\varrho_0\in \Gamma_0(\eta)$ such that $\varrho_0(t_0) = y_0$. The
         open neighbourhood
         \[
          V := (\cap_{i=1}^3 \{t\in T\mid \|\sigma_i(t_0) - \varrho_0(t_0)\| < r_i^{\pr}\})\cap \{t\in T\mid \varrho(t)\in U\}
         \]
         of $t_0$ satisfies $V\subset \{t\in T\mid C(t)\cap U\neq \mset\}$ and the claim about $t\to C(t)$ is proved. We infer that $t\to
         \overline{C(t)}$, $t\in T$, is also lower semi-continuous. Remark that $\overline{C(t)}$ is a non-void closed convex subset of $\F(t)$.
         Moreover, if $y\in overline{C(t)}$ then $\|y\|\leq r_i^{\pr}$
         The set
         \[
          K_1 := \cup_{i=1}^3 \{t\in T\mid \|\sigma_i(t)\|\geq r_i^{\pr}/2\}
         \]
         is compact and there exists an open subset $U\supset K_1$ of $T$ such that $K_2 :+ \overline{U}$ is compact. Let $\rho^{\pr}$ be a
         continuous selection of $t\to \overline{C(t)}$ over $K_2$ given by Theorem \ref{T:selection} and $f : T\to [0,1]$ be a continuous function
         satisfying $f\mid K_1\equiv 1$ and $f\mid (T\setminus U)\equiv 0$. Define
         \[
          \rho(t) :=
          \begin{cases}
             f(t)\cdot \rho^{\pr}(t), \ &\text{if $t\in K_2$},\\
             0_t, \ &\text{if $t\in T\setminus U$}.
          \end{cases}
         \]
         Then $\rho$ is well defined and $\rho\i \Gamma_0(\eta)$. We claim that $\|\sigma_i - \rho\|\leq r_i^{\pr}$, $1\leq i\leq 3$ and we
         shall conclude from this that $\Gamma_0(\eta)$ has the 3-ball-property in $\Gamma_o(\xi)$. If $t\in K_1$ then $\|\sigma_i(t) -
         \rho(t)\| = \|\sigma_i(t) - \rho^{\pr}(t)\|\leq r_i^{\pr}$ since $\rho^{\pr}(t)\in \overline{C(t)}$. If $t\in T\setminus U$ then
         $\|\sigma_i(t) - \rho(t)\| = \|\sigma_i(t)\| < r_i^{\pr}/2$. Finally, if $t\in U\setminus K_1$ then
         \begin{multline}
            \|\sigma_i(t) - \rho(t)\|\leq \|\sigma_i(t) - f(t)\cdot \sigma_i(t)\| + \|f(t)\cdot \sigma_i(t) - f(t)\cdot \rho^{\pr}(t)\|\\
                   (1 - f(t))\|\sigma_i(t)\| + f(t)\|\sigma_i(t) - \rho^{\pr}(t)\|\leq (1 - f(t))r_i^{\pr}/2 + f(t)r_i^{\pr}\leq r_i^{\pr},
         \end{multline}
         and with this the claim is proved.

      \end{proof}

\section{Uniform and locally uniform Banach bundles} \label{S:uniform}

In this section we shall describe a class of Banach bundles that is more general than the class of locally trivial Banach bundles but it is
still quite manageable. This class was introduced in \cite{L} under the guise of continuous fields of Banach spaces. We remind the definition of
the Hausdorff distance between two bounded subsets $A$, $B$ of a metric space $(M,\rho)$:
\[
 d(A,B) := \text{max}\{\sup_{a\in A} \rho(a,B), \sup_{b\in B} \rho(b,A)\}.
\]
It is a metric in the family of all bounded closed subsets of $M$.

Recall that if $X$ is a Banach space then $X_1$ denotes its closed unit ball.

\begin{lem} \label{L:sc}

   Let $X$ be a Banach space, $\mathcal{M}$ the family of its bounded closed subsets endowed with the Hausdorff metric, $T$ a topological space,
   and $t\to X(t)$ a map $\Phi$ from $T$ to the collection of closed subspaces of $X$. If the map $t\to X(t)_1$ from $T$ to $\mathcal{M}$ is continuous
   then $\Phi$ is lower semi-continuous.

\end{lem}

\begin{proof}

   Let $O$ be an open subset of $X$ and suppose that $O\cap X(t_o)\neq \mset$, $t_0\in T$. Pick $x_0\in O\cap X(t_0)$ and let $r > 0$ be so that
   $ B(x_0,r)\subset O$. From the continuity of $t\to X(t)_1$ it follows that there exists a neighbouhood $V$ of $t_0$ in $T$ such that for each
   $t\in V$ there exists $y_t\in X(t)$ satisfying $\|y_t\|\leq \|x_0\|$ and $\|x_0 - y_t\| < r$. Thus, if $t\in V$ then $O\cap X(t)\neq \mset$ and
   the proof is complete.

\end{proof}

We keep the setting and the notations of Lemma \ref{L:sc}; to be consistent with our blanket requirement we shall suppose that $T$ is a
Hausdorff topological space. We make $\{t\}\times X(t)$ into a Banach space by transfering onto it in the obvious way the structure of $X(t)$,
$t\in T$. Put $\E ;= \cup_{t\in T} (\{t\}\times X(t))$ and define $p : \E\to T$ by $p((t,x)) := t$. For an open subset $V$ of $T$ and an open
subset $O$ of $X$ let $\U(V,O) := \{(t,x)\mid t\in V,\ x\in O\cap X(t)\}$.

\begin{prop} \label{P:unif}

   The family of all the sets $\U(V,O)$ when $V$ runs through all the open subsets of $T$ and $O$ runs through all the open subsets of $X$ is a
   base of a topology on $\E$. When $\E$ is endowed with this topology $\xi := (\E,p,T)$ is a continuous Banach bundle.

\end{prop}

\begin{proof}

If $(t,x)\in \U(V_1,O_1)\cap \U(V_2,O_2)$ then
\[
 (t,x)\in \U(V_1\cap V_2,O_1\cap O_2)\subset \U(V_1,O_1)\cap \U(V_2,O_2);
\]
hence the first statement of the proposition. If $V\subset T$ is open then $p^{-1}(V) = \U(V,X)$ so $p$ is continuous. If $V\subset T$ and
$O\subset X$ are open then $p(\U(V,O)) = V\cap \{t\in T\mid O\cap X(t)\neq \mset\}$ so Lemma \ref{L:sc} implies that this set is open in $\E$.
Thus $\xi$ is a Banach bundle. Let $(t_0,x_0)\in \E$, $r > 0$ and set $V := \{t\in T\mid B(x_0,r)\cap X(t)\neq \mset\}$. The continuity of the
norm at $(t_0,x_0)\in \E$ follows from
\begin{multline*}
 (t_0,x_0)\in \U(V,B(x_0,r)) = \{(t,x)\mid t\in V,\ x\in B(x_0,r)\cap X(t)\}\subset \\
                   \{(t,x)\in \E\mid \|(t_0,x_0)\| - r < \|(t,x)\| < \|(t_0,x_0)\| + r\}.
\end{multline*}

\end{proof}

A continuous Banach bundle isomorphic to a Banach bundle as described in Proposition \ref{P:unif} is called a uniform Banach bundle. A
(continuous) Banach bundle $\xi := (\E,p,T)$ is called locally uniform if there exists a family of closed subsets $\{T_{\alpha}\}_{\alpha\in
\A}$ of $T$ such that $\{\mathrm{Int}(T_{\alpha})\}_{\alpha\in \A}$ is an open cover of $T$ and each restriction $\xi\mid T_{\alpha}$ is a
uniform Banach bundle. One should recall that if $T$ is paracompact or locally compact Hausdorff then each subset $T_{\alpha}$ inherits the same
property. If two subspaces of a Banach space have different dimensions then, by \cite{Gu}, the distance between their closed unit balls is at
least $1/2$. Thus if $\xi := (\E,p,T)$ is a locally uniform Banach bundle and for some $t\in T$ the fiber $\E(t)$ is $n$-dimensional, $n <
\infty$, then there exists a neighbourhood $V$ of $t$ in $T$ such that for each $s\in V$ the fiber $\E(s)$ is $n$-dimensional. It follows that
if all the fibers of $\xi$ are finite dimensional and $T$ is locally compact Hausdorff then $\xi$ is locally trivial by \cite[Theorem 18.5]{G}.
On the other hand there exist uniform Banach bundles with infinite dimensional fibers that are not locally trivial. Indeed, Kadets \cite{K}
constructed a Banach space $X$ having subspaces $\{X_n\}_{n=1}^{\infty}$ and $Y$ such that each subspace $X_n$ is isometric to the space
$\ell^{p_n}$ for a sequence $\{p_n\}$ that increases to $2$ and $Y$ is isometric to $\ell^2\}$. Moreover, the the sequence of the closed unit
balls of the subspaces $X_n$ converges to the closed unit ball of $Y$\footnote{The author is grateful to Professor E. Gluskin for providing him
this reference.}. Thus we have here an example of a uniform Banach bundle whose base space is $\mathbb{N}\cup \{\infty\}$ with the natural
topology that is not isomorphic to a locally trivial bundle since, as it is well known, the spaces $\ell^{p^{\pr}}$ and $\ell^{p^{\pr \pr}}$ are
not isomorphic if $p^{\pr}\neq p^{\pr \pr}$.

We describe now a way to construct locally uniform Banach bundles. Let $T$ be a regular topological space, $\{V_{\alpha}\}_{\alpha\in \A}$ a
cover of $T$ with open non-void subsets, $\{X_{\alpha}\}_{\alpha}$ a family of Banach spaces such that whenever $V_{\alpha}\cap V_{\beta}\neq
\mset$, $X_{\alpha}\cap X_{\beta}$ is a non-trivial closed subspace of $X_{\alpha}$ and of $X_{\beta}$ on which the norms $\|\cdot\|_{\alpha}$
and $\|\cdot\|_{\beta}$ coincide. Denote by $\mathcal{M}_{\alpha}$ the hyperspace of all bounded closed subsets of $X_{\alpha}$ endowed with the
Hausdorff metric. Suppose that for each $t\in V_{\alpha}$, $\alpha\in \A$, $X(t)$ is a closed subspace of $X_{\alpha}$ and $t\to X(t)_1$ is
continuous as a map from $V_{\alpha}$ to $\mathcal{M}_{\alpha}$. Of course, if $t\in V_{\alpha}\cap V_{\beta}$ then $X(t)\subset X_{\alpha}\cap
X_{\beta}$. Put
\[
 \E_{\alpha} := \{(t,x)\mid t\in V_{\alpha}, x\in X(t)\}
\]
and give $\E_{\alpha}$ the topology specified in \ref{P:unif}. That is, each pair of open subsets $V\subset V_{\alpha}$ and $O\subset
X_{\alpha}$ determines an open subset
\[
 \U(V,O) := \{(t,x)\in \E_{\alpha}\mid t\in V,\ x\in O\cap X(t)\}.
\]
Observe that if $V_{\alpha}\cap V_{\beta}\neq \mset$ then on $\E_{\alpha}\cap \E_{\beta}$ the two relative topologies coincide. Set now $\E :=
\cup_{\alpha\in \A} \E_{\alpha}$ and define a topology in $\E$ as follows: $\U\subset \E$ is open if $\U\cap \E_{\alpha}$ is open in
$\E_{\alpha}$ for every $\alpha\in \A$. Let $p : \E\to T$ be the natural map.

\begin{prop} \label{P:lunif}

   With the above notations $\xi := (\E,p,T)$ is a locally uniform Banach bundle.

\end{prop}

\begin{proof}

   For $t\in V_{\alpha}$ let $A_{t,\alpha}$ be a closed set such that $t\in \mathrm{Int}(A_{t,\alpha})\subset A_{t,\alpha}\subset V_{\alpha}$.
   Clearly the restriction of $\xi$ to $A_{t,\alpha}$ is a uniform Banach bundle.

\end{proof}

Continuing with the above notations, if $Y(t)$ is a closed subspace of $X(t)$, $t\in T$, and $t\to Y(t)_1$ is continuous on $T$ then, with $\F
:= \{(t,x)\mid t\in T, x\in Y(t)\}$, $\zeta := (\F,p\mid \F, T)$ is a (locally uniform) Banach subbundle of $\xi$. Indeed, if $V$ is an open
subset of $V_{\alpha}$ and $O$ is an open subset of $X_{\alpha}$ then
\begin{multline*}
 \{(t,x)\mid t\in V,\ x\in O\cap X(t)\}\cap \{(t,x)\mid t\in V_{\alpha},\ x\in Y(t)\} =\\
             \{(t,x)\mid t\in V,\ x\in O\cap Y(t)\}
\end{multline*}
is a typical open subset of $\F$. Now it is clear that the Banach bundle constructed in Proposition \ref{P:unif} is a Banach subbundle of a
trivial Banach bundle. Indeed, with the notations used there, let $\E^{\pr} := T\times X$ and $p^{\pr}$ the projection on $T$. Then $\xi$ is by
the preceding remarks a Banach subbundle of $\xi^{\pr} := (\E^{\pr},p^{\pr},T)$. Note that $\E$ is a closed subset of $\E^{\pr}$. To see this,
let $\{(t_{\iota},x_{\iota})\}_{\iota\in \mathcal{I}}$ be a net in $\E$ converging to $(t,x)\in T\times X$. There is no loss of generality if we
suppose that $x_{\iota}\in B_X(0_X,1)$, $\iota\in \mathcal{I}$. For each $\iota$ there exists $y_{\iota}\in \E(t)$ such that $\|x_{\iota} -
y_{\iota}\|\leq d(\E(t_{\iota})_1,\E(t)_1)$, $d$ being the Hausdorff metric. Now $d(\E(t_{\iota})_1,\E(t)_1)\to 0$ so $(t,x)\in \E$. On the
other hand, not every Banach subbundle of a trivial Banach bundle which has a closed bundle space is a uniform Banach bundle.

\begin{ex} \label{E:closed}
   Let $T :=\mathbb{N}\cup \{\infty\}$ with the usual topology, $c_0$ the Banach space of all null-convergent sequences of scalars, $\E := T\times
   c_0$ and $p$ the projection of $\E$ onto $T$. We shall construct a Banach subbundle of the trivial Banach bundle $\xi := (\E,p,T)$. Denote by
   $e_n\in c_0$ the sequence $\{\delta_n^k\}_{k=1}^{\infty}$ and let $\F(n)$ be the linear subspace of $\{n\}\times c_0$ spanned by $(n,e_1)$ and
   $(n,e_n)$. We let $\F(\infty)$ be the one-dimensional space spanned by $(\infty,e_1)$. With $\F := \cup_{t\in T} \F(t)$ the Banach subbundle
   $\zeta := (\F,p\mid\F,T)$ of $\xi$ has a closed bundle space. However, the bundle $\zeta$ is not locally uniform because of the drop in
   dimension at $\infty$.

\end{ex}

We have seen in Theorem \ref{T:Gamma0} that for a bundle map it is of interest to have a continuous norm in the image bundle. We discuss now
conditions that insure that the norm in a quotient bundle is continuous. It may happen that a quotient bundle of a continuous Banach bundle has
a discontinuous norm as the following example shows.

\begin{ex}

   Let $T := \mathbb{N}\cup \{\infty\}$ with the usual topology, $\xi := (T\times \mathbb{K},p,T)$ the one-dimensional trivial bundle over
   $T$, and $\eta$ its subbundle whose bundle space is $\F := (\mathbb{N}\times \mathbb{K})\cup (\{\infty\}\times \{0\})$. Then the quotient
   bundle space is $(\mathbb{N}\times \{0\})\cup (\{\infty\}\times \mathbb{K})$ with a non-Hausdorff topology. Hence the quotient bundle is not
   a continuous Banach bundle.

\end{ex}

\begin{qu}

   Suppose $\eta$ is a quotient bundle of a continuous Banach bundle and the bundle space of $\eta$ is Hausdorff. Must $\eta$ be a continuous
   Banach bundle?

\end{qu}

\begin{prop} \label{P:qcont2}

   Let $\xi := (\E,p,T)$ be a locally uniform Banach bundle and $\xi^{\pr} := (\E^{\pr},p\mid \E^{\pr},T)$ a locally uniform subbundle. Then the norm on the quotient
   bundle $\eta := (\E/\E^{\pr},p^{\pr},T)$ is continuous.

\end{prop}

\begin{proof}

   A moment's reflection shows that it is enough to consider only the case when $\xi$ and $\xi^{\pr}$ are uniform. Thus we shall suppose that $\xi$ is a
   bundle as described in Proposition \ref{P:unif} keeping the notation used there. We shall also suppose that $Y(t)$ is a closed subspace of
   $X(t)$ and that the map $t\to Y(t)_1$, $t\in T$, is continuous into the the space of all closed subsets of the ambient space $X$. Then $\E ;+
   \cup_{t\in T} (\{t\}\times X(t))$ and $\E^{\pr} ;= \cup_{t\in T} (\{t\}\times Y(t))$. Denote by $q$ the quotient map from $\E$ to
   $\E/\E^{\pr}$ and let $\alpha\geq 0$; we are going to show that $O := \{y\in \E/\E^{\pr}\mid \|y\| > \alpha\}$ is open in $\E/\E^{\pr}$. Suppose then that
   $y_0\in O$ and choose numbers $\beta$ and $\varepsilon$ satisfying $\alpha < \beta < \|y_0\|$ and $0 < \varepsilon < \min\{1/2,
   1/3(\|y_0\| - \beta)\}$. Let $(t_0,x_0)\in \E$ with $q((t_0,x_0)) = y_0$ and put $r := 2\|x_0\| + 2$. There exists a neighbourhood $V$ of $t_0$
   in $T$ such that $d(Y(t)_r,Y(t_0)_r) < \varepsilon$ if $t\in V$. The set $\U(V,B(x_0,\varepsilon)) := \{(t,x)\mid t\in V,\ x\in
   B(x_0,\varepsilon)\cap X(t)\}$
   is an open neighbourhood of $(t_0,x_0)$ and $q(\U(V,B(x_0,\varepsilon)))$ is an open neighbourhood of $y_0$, $q$ being an open map. We shall show
   that $q(\U(V,B(x_0,\varepsilon)))\subset O$ and this will establish our claim about $O$ thereby completing the proof.

   Let $(t,x)\in \U(V,B(x_0,\epsilon))$ and denote $y := q(x)$. Our goal is to show that $\|y\| > \alpha$. There exists $\tilde{x}\in Y(t)$ such
   that $\|y\|\leq \|x - \tilde{x}\| < \|y\| + \varepsilon$. We have $\|\tilde{x}\|\leq r$. Indeed, if we assume by contradiction that
   $\|\tilde{x}\| > r = 2\|x_0\| + 2$ then $\|\tilde{x}\| > 2\|x - (x - x_0)\| + 2 > 2(\|x\| - \varepsilon) + 2 > 2\|x\| + 1$. But then
   \[
    \|x - \tilde{x}\|\geq \|\tilde{x}\| - \|x\| > \|x\| + 1 > \|y\| + \varepsilon,
   \]
   a contradiction.
   Now, since $t\in V$ and $\tilde{x}\in Y(t)_r$, there exists $\tilde{x}_0\in Y(t_0)_r$ such that $\|\tilde{x} - \tilde{x}_0\| <
   \varepsilon$. We have
   \begin{multline}
      \|y\| > \|x - \tilde{x}\| - \varepsilon\geq \|(\tilde{x} - \tilde{x}_0) + (\tilde{x}_0 - x_0)\| - \|x_0 - x\| - \varepsilon\geq \\
               \|\tilde{x}_0 - x_0\| - \|\tilde{x} - \tilde{x}_0\| - \|x_0 - x\| - \varepsilon\geq \|y_0\| - 3\varepsilon > \beta >\alpha
   \end{multline}
   and we are done.

\end{proof}

The quotient bundle of a uniform Banach bundle may have a continuous norm even when the subbundle that produces it is not a locally uniform
subbundle as the following example shows.

\begin{ex}

   As before we denote by $c_0$ the Banach space of all null converging sequences of scalars and $e_n := \{\delta_n^k\}_{k=1}^{\infty}$, $n\geq
   1$. Let $T := \mathbb{N}\cup \{\infty\}$ with the usual topology, $\E := T\times c_0$, and $p$ be the projection onto $T$. We let $X_n$ be the
   subspace of $c_0$ spanned by $\{e_l\}_{l=1}^n$, $\E^{\pr}(n) := \{n\}\times X_n$, and $\E^{\pr}(\infty) := \E(\infty)$. We affirm that $\xi^{\pr} :=
   (\E^{\pr},p\mid \E^{\pr},T)$ is a subbundle of the trivial bundle $\xi := (\E,p,T)$. Indeed, suppose $U\subset \E$ is open and $(\infty,x)\in
   U\cap \E^{\pr}(\infty)$. There are $r > 0$ and $m\in \mathbb{N}$ with the property that if $y\in c_0$, $\|y - x\| < r$, and $n\geq m$ then $(n,y)\in
   U$. Let $k\in \mathbb{N}$ be such that $\bar{x} ; = (x_1,\dots,x_k,0\dots)$ satisfies $\|\bar{x} - x\| < r$. Then if $n > \max\{m,k\}$ one has
   $(n,\bar{x})\in U\cap \E^{\pr}(n)$. Thus we found a neighbourhood of $\infty$ included in $p(U)$ and the claim is proved.

   Clearly $\xi^{\pr}$ is not a locally uniform subbundle because of the disparity in the dimension of the fiber at $\infty$ with the dimensions
   of the other fibers in any neighbourhood of $\infty$. It remains to show that the norm in the quotient bundle is continuous. Only the
   continuity at the points of the fiber at $\infty$ must be proven; but since the fiber at $\infty$ of the quotient bundle is trivial this fact
   is obvious.

\end{ex}

\appendix
\section{Some topological properties of the bundle space}

We shall now record several topological properties of the bundle space of a Banach bundle. In the following $\xi := (\E,p,T)$ will denote a
Banach bundle.

The first result treats the complete regularity of the bundle space. A result with the same conclusion but under different hypotheses is
\cite[Lemma 2.1.6]{Gut}.

\begin{prop} \label{P:cr}

   If $\xi$ is a continuous Banach bundle and $T$ is locally paracompact then $\E$ is completely regular.

\end{prop}

\begin{proof}

   First remark that $\E$ is a Hausdorff space since the norm is continuous on it. Let $A$ be a closed subset of $\E$ and $x_0\in \E$
   with $x_0\notin A$. There are an open neighbourhood $V$ of $t_0 := p(x_0)$ in $T$ such that
   $\overline{V}$ is paracompact and a continuous function $h : T\to [0,1]$ satisfying $h(t_0) = 1$ and $h\mid T\setminus V \equiv 0$. Corollary
   \ref{C:ext} yields a section $\varphi$ of $\xi\mid \overline{V}$ such that $\varphi(t_0) = x_0$; there exists $\alpha > 0$ such that $\{x\in
   \E\mid p(x)\in V,\ \|x - \varphi(p(x))\| < \alpha\}\cap A = \mset$. Let $\mu : [0,\infty)\to [0,1]$ be given by
   \[
      \mu(s) :=
      \begin{cases}
        1 - s/\alpha,   &\text{if $0\leq s\leq \alpha$,}\\
        0,              &\text{if $\alpha < s$}.
      \end{cases}
   \]
   The function $f: \E\to [0,1]$ defined by
   \begin{equation*}
      f(x) :=
      \begin{cases}
         \mu(\|x - \varphi(p(x))\|)h(p(x)), &\text{if $p(x)\in V$,}\\
         0,                                 &\text{if $p(x)\notin V$}
      \end{cases}
   \end{equation*}
   satisfies $f(x_0) = 1$ and $f\mid A \equiv 0$ as needed.

\end{proof}

We proceed now to discuss a setting in which $\E$ is a Baire space. We begin with some definitions and results from \cite{O}. A family $\B$ of
non-void open subsets of a topological space is call ed a pseudo-base of the space if every open non-void subset contains an element of $\B$. A
topological space $Z$ is called quasi-regular if each non-void open set of $Z$ contains the closure of some non-void open set. A topological
space $Z$ is called pseudo-complete if it is quasi-regular and there exists a sequence $\{\B_n\}$ of pseudo-bases of $Z$ with the property that
if $U_n\in \B_n$ and $\overline{U_{n+1}}\subset U_n$, $n = 1,2,\ldots$ then $\cap_n U_n\neq \mset$. It is easily seen that complete metric
spaces and locally compact Hausdorff spaces are pseudo-complete. The cartesian product of any family of pseudo-complete spaces is
pseudo-complete, see \cite[Theorem 6]{O}. Any pseudo-complete-space is a Baire space, see \cite[5.1]{O}.

\begin{prop}

   Suppose that the base space $T$ is pseudo-complete and locally paracompact. Then $\E$ is a Baire space.

\end{prop}

\begin{proof}

   Let $\{\B_n\}$ be sequence of pseudo-bases of $T$ as in the definition of the pseudo-completeness. Let $\{G_n\}$ be a sequence of open dense
   subsets of $\E$ and $G$ an open subset of $\E$. By the assumption of local paracompactness of $T$ and \cite[p. 10]{DG} there are an open
   subset $V_1$ of $T$, a section $\varphi_1$ of $\xi$ over $V_1$, and a number $\varepsilon_1$, $0 < \varepsilon_1 < 2^{-1}$, such that
   \[
    U_1 := \{x\in \E\mid p(x)\in V_1,\ \|x - \varphi_1(p(x))\|\leq \varepsilon_1\}\subset G\cap G_1.
   \]
   From the assumption of pseudo-completeness of $T$ we infer that there exists $W_1\in \B_1$ such that $W_1\subset \overline{W_1}\subset V_1$.
   Now there are an open subset $V_2\subset W_1$ of $T$, a section $\varphi_2$ of $\xi$ over $V_2\subset W_1$, and a number $\varepsilon_2$, $0 <
   \varepsilon_2 < 2^{-2}$, such that
   \[
    U_2 := \{x\in \E\mid p(x)\in V_2,\ \|x - \varphi_2(p(x))\|\leq \varepsilon_2\}\subset U_1\cap G_2.
   \]
   There exists $W_2\in \B_2$ such that $W_2\subset \overline{W_2}\subset V_2$. An obvious induction process shows that for every natural number
   $n\geq 2$ there exist an open subset $V_n\subset W_{n-1}$ of $T$, a section $\varphi_n$ of $\xi$ over $V_n$, and a number $\varepsilon_n$, $0 <
   \varepsilon_n < 2^{-n}$, such that
   \[
    U_n := \{x\in \E\mid p(x)\in V_n,\ \|x - \varphi_n(p(x))\|\leq \varepsilon_n\}\subset U_{n-1}\cap G_n
   \]
   and there exists $W_n\in \B_n$ such that $W_n\subset \overline{W_n}\subset V_n$.

   We have $\cap_{n=1}^{\infty} V_n = \cap_{n=1}^{\infty} W_n\neq \mset$. Pick $t_0\in \cap_n V_n$. From $\{\varphi_n(t_0) - \varphi_{n-1}(t_0)\|\leq \varepsilon_{n-1} < 2^{n-1}$
   we infer that $\{\varphi_n(t_0)\}$ is a Cauchy sequence in the Banach space $p^{-1}(t_0)$; it converges to a point $x_0$ in this space. The
   inequality $\|x_0 - \varphi_n(t_0)\|\leq \varepsilon_n$, $n = 1,2,\ldots$, shows that $x_0\in \cap_{n=1}^{\infty} U_n\subset
   G\cap\bigcap_{n=1}^{ \infty} G_n$ and the proof is complete.

\end{proof}

We have seen in Propositions \ref{P:BG1} and \ref{P:hom} that it is of interest to know if a given Banach bundle has a paracompact bundle space.
We turn now to examining some cases when $\E$ is paracompact or metrizable. Suppose $\xi$ is a uniform Banach bundle.Then, as remarked above in
the paragraph preceding Example \ref{E:closed}, there exists a Banach space $X$ such that $\E$ is a closed subset of $T\times X$. If $T$ is
regular and $\sigma-compact$ or paracompact and perfectly normal then $T\times X$ is paracompact by \cite[Propositions 4 and 5]{M0}. Hence in
these cases $\E$ is paracompact by \cite[p. 218]{E}. With some more work one can obtain the same conclusion when $\xi$ is a locally uniform
Banach bundle. Indeed, let $\{T_{\alpha}\}$ be a family of closed subsets of $T$ such that $\{\mathrm{Int}(T_{\alpha})\}$ is an open cover of
$T$ with non-void subsets and the restriction of $\xi$ to each $T_{\alpha}$ is a uniform Banach bundle. In both cases mentioned above $T$ is
paracompact so we may suppose that the family of sets $\{T_{\alpha}\}$ is locally finite. Each $T_{\alpha}$ inherits the properties of $T$ we
consider so each $p^{-1}(T_{\alpha})$ is paracompact. But $\{p^{-1}(T_{\alpha})\}$ is a locally finite closed cover of $\E$ so $\E$ is
paracompact.

Now if $\xi$ is a uniform Banach bundle with a metrizable base space $T$ then $\E$ is metrizable by being a subset of a metrizable cartesian
product. Again this remains true for a locally uniform Banach bundle $\xi$. If $\{T_{\alpha}\}$ is a locally finite closed cover of $T$ as in
the previous paragraph then $\{p^{-1}(T_{\alpha})\}$ is a locally finite closed cover of $\E$ with metrizable subspaces. Then a result of
Nagata, see \cite[p. 204]{E}, implies that $\E$ is metrizable.

\begin{prop}

   Let $\xi$ be a continuous Banach bundle with a regular base space $T$ that has a countable base. Moreover, suppose that the Banach space
   $\Gamma_b(\xi)$ is separable. Then $\E$ is regular and has a countable base.

\end{prop}

\begin{proof}

   Remark first that $T$ is a metrizable space therefore, by Proposition \ref{P:cr}, $\E$ is completely regular. Now, say that $\{V_m\}$ is a
   countable base of $T$ and $\{\varphi_n\}$ is a sequence dense in $\Gamma_b(\xi)$. For $m,n,k\in \mathbb{N}$ denote
   \[
    \U(m,n,k) := \{x\in \E\mid p(x)\in V_m,\ \|x - \varphi_n(p(x))\| < 1/k.
   \]
   We shall show that $\{\U(m,n,k)\}$ is a base of $\E$. Let $U\subset \E$ be open and $x_0\in U$. The Banach bundle $\xi$ is full so by \cite[p. 10]{DG}
   there exist $m\in \mathbb{N}$, $\varphi\in \Gamma_b(\xi)$, and $\delta > 0$ such that $p(x_0)\in V_m$, $\varphi(p(x_0) = x_0$, and
   \[
    U^{\pr} := \{x\in \E\mid p(x)\in V_m,\ \|x - \varphi(p(x))\| < \delta\}\subset U.
   \]
   Choose $k\in \mathbb{N}$ such that $1/k < \delta/2$ and
   $\varphi_n$ such that $\|\varphi - \varphi_n\| < 1/k$. It is now routine to check that $x_0\in U^{\pr}\subset \U(m,n,k)$.

\end{proof}

\bibliographystyle{amsplain}
\bibliography{}

\end{document}